\documentclass[a4paper, 11pt]{scrartcl}
  \usepackage[utf8]{inputenc}
	\usepackage{csquotes}
  \usepackage[english]{babel}

  \usepackage{paralist}
  \usepackage{mathtools}
  \usepackage{amssymb, amsthm}
  \usepackage{microtype}
  \usepackage{enumerate}
  \usepackage{mathrsfs}
  \usepackage{xcolor}
  \usepackage{graphicx}
  \usepackage{paralist}
  \usepackage[normalem]{ulem}
  \usepackage{tikz}
  \usetikzlibrary{matrix,arrows,positioning}
  \usepackage[square,numbers]{natbib}
  \usepackage{hyperref}
  \usepackage{cleveref}

  \theoremstyle{plain}
  \newtheorem{theorem}{Theorem}[section]
  \newtheorem{lemma}[theorem]{Lemma}
  \newtheorem{prop}[theorem]{Proposition}
  \newtheorem{corollary}[theorem]{Corollary} 
    
  \theoremstyle{definition}
  \newtheorem{defn}[theorem]{Definition}
  
  \theoremstyle{remark}
  \newtheorem{remark}[theorem]{Remark}

  \DeclarePairedDelimiter{\abs}{\lvert}{\rvert}
  
  \DeclarePairedDelimiter{\sprod}{\langle}{\rangle}
  \DeclarePairedDelimiter{\sn}{[\kern-0.2em[}{]\kern-0.2em]}
  \DeclarePairedDelimiter{\FP}{\{}{\,|\kern-0.2em \}}
  \DeclarePairedDelimiter{\FK}{[}{\,|\kern-0.2em ]}
  \DeclarePairedDelimiter{\p}{\{}{\}}

  \newcommand{\I}{\mathrm {i}}
  \newcommand{\E}{\mathrm {e}}
  \newcommand{\N}{\mathbb{N}}
  \newcommand{\K}{\mathbb{K}}
  
  \newcommand{\R}{\mathbb{R}}
  
  \newcommand{\C}{\mathbb{C}}
  
  \newcommand{\Z}{\mathbb{Z}}
  
  \newcommand{\M}{\mathscr{M}}
  \newcommand{\HH}{\mathbf{H\!H}}
  \renewcommand{\d}{\mathrm{d}}
  
  \renewcommand{\H}{\text{H}}
  
  \renewcommand{\phi}{\varphi}
  \newcommand{\Alt}{\operatorname{Alt}}
  \newcommand{\GL}{\operatorname{GL}}

  \newcommand{\sign}{\operatorname{sign}}

  \newcommand{\rang}{\operatorname{rank}}
  
  \newcommand{\pr}{\operatorname{pr}}

  \newcommand{\im}{\operatorname{im}}
  \newcommand{\id}{\operatorname{id}}

  \newcommand{\CC}{\mathscr{C}}
  \newcommand{\A}{\mathscr{A}}

  \newcommand{\del}{\partial}
  \renewcommand{\O}{\mathcal{O}}
  \newcommand{\factor}[2]{\left.\raisebox{.2em}{$#1$}\middle/\raisebox{-.2em}{$#2$}\right.}
  
  \newcommand{\ph}{[[\lambda]]}
  \newcommand{\CCinf}{\CC^\infty}
  \newcommand{\X}{\mathfrak{X}}

  \newcommand{\HC}{\mathbf{HC}}
  \newcommand{\Hom}{\operatorname{Hom}}

  \newcommand{\DO}{\operatorname{DiffOp}}
  \newcommand{\DOver}{\operatorname{DiffOp_{ver}}}
  \newcommand{\DOhor}{\operatorname{DiffOp_{hor}}}
  
  \renewcommand{\id}{\operatorname{id}}

  \newcommand{\opp}{\mathrm{opp}}
  \newcommand{\diff}{\mathrm{diff}}

  \newcommand{\cont}{\mathrm{cont}}
  
  \newcommand{\Sym}{\mathcal{S}}
  \newcommand{\tilbullet}{\mathbin{\tilde \bullet}}
  \newcommand{\prbullet}{\mathbin{\bullet'}}
  \newcommand{\prol}{\operatorname{prol}}

  \newcommand{\otimeshat}{\mathbin{\hat\otimes}}

  \newcommand{\No}{\mathfrak{N}}

  \allowdisplaybreaks

	\title{Bimodule Deformation of fibered manifolds and the HKR theorem}
	\date{June 1, 2018}
  \author{\textsc{Benedikt Hurle\thanks{benedikt.hurle@uha.fr}}\\[2mm] 
	\normalsize Universit\'e de Haute Alsace, Mulhouse (France) \\
	\normalsize IRIMAS, Département de Mathématiques\\ 
	\vspace{-5mm}}
\begin{document}

\maketitle

\begin{abstract}
 	\textbf{Abstract}	We first  want to consider the formal deformation of a fibered manifold  $P \rightarrow M$ as a (bi-)module or subalgebra, where $M$ has a given differential star product. The module case has already been dealt with in \citep{art,weisphd}. Consequently  we want to find obstructions for the existence of a bimodule or subalgebra, which turns out to be the curvature of the fiber bundle. Since the order by order construction of this structures amounts to solving equations in the Hochschild cohomology $\HH^\bullet(\CCinf(M),\DO(P))$, we proceed to computing this cohomology and also the very similar cohomology $\HH^\bullet(\CCinf(M),\CCinf(P))$ for the case of a smooth map $P \xrightarrow{\pr} M$ such that $\pr(P)$ is a closed submanifold of $M$.
						
\end{abstract}

\tableofcontents

\section*{Acknowledgment}
This work  is mostly part of my master thesis, which I did at Würzburg University.
I want to thank Stefan Waldmann, my advisor then, very much for  very many  helpful discussions and  ideas he gave me. He also suggested this topic to me.

I also want to thank Martin Bordemann for his comments and discussions with him.

\section{Introduction}

The aim of deformation quantization is to get from a classical physical system  described by a Poisson or symplectic manifold $M$ to a quantum theory, which has  as a classical limit this given system. For this one introduces a star product on the the formal power series of smooth functions $\CCinf(M)\ph$ on $M$. A star product is an associative but normally noncommutative product. From this one obtains the classical Poisson bracket by taking the limit $\hbar \rightarrow 0$ of the star commutator $\frac{\I}{\hbar} [\cdot,\cdot]_\star$.  This method has been introduced in \citep{bayen}.

Another idea, which is more recent, is to deform classical field theories by replacing the commutative algebra of functions on the spacetime manifold by a noncommutative one. The idea here is to deform the commutator of the coordinate functions, which is classical $[x_i,x_j] =0$, to something non-zero. There are many different approaches to this coming form theoretical physics, which lead to noncommutative field theories, see \citep{aschieri,wess,fredenhagen, douglas,schupp}.  Most of these approaches  only consider the case of $\R^4$ with a Weyl-Moyal product, however a more general approach is also needed. On the other hand there are quite concrete solutions to the corresponding noncommutative Einstein equations \citep{schenkel}.
This leads to what is called noncommutative geometry, which also has been studied from a more mathematical point of view, see e.g. \citep{connes}.

If one wants to deform a field theory in this way, one also needs to deform bundles over the spacetime manifold, especially principal bundles and  vector bundles, because this is where the fields, the connection or  the curvature (which is the field strength tensor) live. 
There are several approaches of doing this. One is by Connes, which uses  a so called spectral triple, see \citep{connes}. It was shown by Hawkins in \citep{hawkins} that this approach only works in some  situations.

Here we want to consider what happens in the context of deformation quantization. For this we consider the quite general situation of a fibered manifold, which can be specialized the principal bundles and other  cases.  
The weakest way is to deform those into a module, which has be done in \citep{weisphd,art}, and always works. But for many applications this seems not enough. For example to write  the Leibniz rule $\d(fa) = (\d f) a + f \d a$ with $f$ in some bundle over $M$ and $a \in \CCinf(M)$ in this form one would already need a bimodule. For the case of vector bundles this was also considered e.g. in \citep{waldmanne}. One can also use  a Drinfeld twist, see e.g. \citep{aschieri}, to do noncommutative geometry. But also here obstructions exist \citep{weber}.  Given a Drinfeld twist one can also define a star product, so to some extend what we do is more general. 
Also in the context of noncommutative geometry often Hopf-Galois extension are considered as a generalization of principal bundles, e.g. \cite{brezinski,MR2175995}, but  here one cannot deal with symplectic bases in general.

So the aim of this paper is to investigate under which conditions such bimodule structures for a fibered manifold $P \to M$ exist. It turns out that especially for the symplectic case there are strong obstructions and it is only possibly to get such a bimodule in very special cases, e.g.  if the bundle is trivial or there exists a flat connection on $P$. To be precise one gets the structure of a Poisson module on $\CCinf(P)$ over the Poisson algebra $\CCinf(M)$.  This can be used te define a morphism of differential operators $\DO(M) \to \DO(P)$, which respects the fiber projection.  This can be seen as a generalization  of a flat lift of the vector fields on $M$.

Since the order by order construction of these module and bimodule structures is equivalent to solving  equations in the Hochschild cohomology $\HH^\bullet(\CCinf(M),\DO(P))$, the second aim of this paper is to compute some of these  cohomologies, namely $\HH^\bullet(\CCinf(M),\allowbreak \CCinf(N))$ and $\HH^\bullet(\CCinf(M),\allowbreak \DO(N))$ for a sufficiently nice map $\pr:N\rightarrow M$. To be precise here we consider the differential or continuous Hochschild cohomology and not the purely algebraic one. This gives us - among other things - a generalization of the well known Hochschild-Kostant-Rosenberg theorem.
In fact we have 

\begin{theorem} 
	Let $N \xrightarrow{\pr} M$ be such that $ \pr(N)$ is a closed submanifold of $M$ then
	\begin{equation*}
		\HH^\bullet_\diff(\CCinf(M),\DO(N)) \cong \factor{\X^\bullet(M)|_{\pr(N)}}{\sprod{\X(\pr(N))}} \otimes_{\CCinf(M)} \DOver(N)
	\end{equation*}
	as $\CCinf(M)$-bimodule, where $\X^\bullet(M)$ denotes the set of vector fields on $M$ and $\sprod{x}$ denotes the ideal generated by $x$.
\end{theorem}

The paper is structured as follows: In the first section we we recall the basics of deformation quantization.
In the second section we first summarize the results from \citep{weisphd} and \citep{art} on module deformation and its relation to Hochschild cohomology. We proceed in finding the obstruction for a bimodule deformation, which in the symplectic case turns out to be the existence of a flat lift.  
In the last section, we compute the Hochschild cohomology $\HH^\bullet(\CCinf(M),\CCinf(N))$ and $\HH^\bullet(\CCinf(M),\DO(N))$ for a map $\pr$ between two manifolds $N$ and $M$, such that $\pr(N)$ is a closed submanifold of $M$, with the bimodule structure given by the pullback along $\pr$. This is done by using the Koszul complex of a convex set in $\R^n$, which we also define in this section. Computing this cohomology is useful,  because the vanishing of it in certain cases proves the fact that every fiber bundle can be deformed into a module and it also shows that, in the case of a bimodule, there are in general problems to be expected due to fact that the Hochschild cohomology is non trivial.

\section{Deformation of fibered manifolds}
\subsection{Star products}

We want to recall some basic definitions and facts about the deformation  quantization of smooth manifolds and star products.  

\begin{defn}[Star product]
	A (formal)  star product $\star$ on a manifold $M$ is a bilinear associative operation $\CCinf(M)\ph \times \CCinf(M)\ph \rightarrow \CCinf(M)\ph$ satisfying the following properties for all $f,g \in \CCinf(M)$:
	\begin{itemize}
		\item $ 1 \star f = f \star 1  = f$,
		\item $f \star g = f\cdot g + \O(\lambda)$,
		\item $f \star g = \sum_{k=0}^\infty C_k(f,g) \lambda^k$,
	\end{itemize}
	with bilinear operators $C_k$. We assume that all $C_k$ are bidifferential operators. 
	It is called natural if every $C_k$ is a differential operator of order $k$.
\end{defn}

We define the star commutator for $a,b \in \CCinf(M)\ph$ by 
$[a,b]_\star = a \star b - b\star a$.
As usual the star commutator satisfies the Leibniz and Jacobi-identity and so gives a non-commutative Poisson algebra. Also the adjoined action is a derivation of $\CCinf(M)\ph$ for all $a \in \CCinf(M)\ph$.
	
It is well known that the first order term of a star product defines a Poisson bracket as follows 
\begin{equation}
	\p{f,g} =  \frac{\I}{2\lambda} [f,g] |_{\lambda=0} \text{ for } f,g \in \CCinf(M).
\end{equation}

\begin{defn}[Equivalence of star products \citep{bayen}] 
	Two star products $\star$, $\star'$ are called equivalent if there exists a formal power series of differential operators 
	$T= \id + \sum_{k=1}^\infty T_k \lambda^k$, with $T(1) =1$ such that 
	\begin{equation}
		T(f) \star T(g) = T(f \star' g)
	\end{equation} 
\end{defn}
The operator $T$ in the above definition is always invertible and indeed, given a star product $\star$,
$f \star' g := T^{-1}(T(f) \star T(g))$ always gives a new equivalent star product. We recall:

\begin{lemma}
	Two equivalent star products give rise to the same Poisson bracket.
\end{lemma}

\subsection{Module deformations}

We want to find criteria, for which star products on a manifold $M$ and fibered manifolds $P$ over $M$ it is possible or not  to get a deformation of $\CCinf(P)$. We consider three different possibilities namely deformation as a module, as a bimodule and as a subalgebra. Here each version is stronger than the previous. For the module case essentially everything is known and works well \citep{art}. For the other cases this is not true. Here we give some obstructions, why things cannot always work, but also some examples where it works well.

For the convenience of the reader we recall some definitions and facts about module deformations from \citep{art,weisphd}, for proofs see there.

\begin{defn}
	A (left) module deformation of fibered manifold $P \xrightarrow{\pr} M$, where $M$ carries a star product $\star$, is a $(\CCinf(M)\ph,\star)$-left module  structure $\bullet$ on $\CCinf(P)\ph$, such that
	\begin{equation}
		a \bullet f = \pr^*a f + \sum_{k=1}^\infty \lambda^k L_k(a,f) = \sum_{k=0}^\infty \lambda^k L_k(a,f),
	\end{equation}
	where the $L_k \in \DO^\bullet(\CCinf(M),\CCinf(P);\CCinf(P))$ are bidifferential operators.
	A module deformation is called fiber preserving if $a \bullet \pr^* b = \pr^* (a \star b)$. It is called natural if all  $L_k$ are differential operators of order up to $k$ on $M$ and $P$. 
\end{defn}

The local form of a $L_k$ is given by
\begin{equation}
	L_k(a,f) = \sum_{I,J} \pr^*(\del_I a) L^{I,J}_k (\del_J f),
\end{equation}
where $I,J$ are multiindices and $L^{I,J}_k \in \CCinf(P)$ are coefficient functions.

Being fiber preserving is equivalent to $a \bullet 1 = \pr^* a$ for all $ a \in \CCinf(M)$, since then $a \bullet \pr^* b = a \bullet (b \bullet 1) = (a \star b) \bullet 1 = \pr^* (a \star b)$.

Similarly one can define a right module deformation. In this case we write $f \bullet a = \pr^*a f + \sum_{k=1}^\infty \lambda^k R_k(f,a)$.

It is also possible to define a module deformation  for an arbitrary map $\pr : P \rightarrow M$ in a similar way.

\begin{defn}
	Two module deformations $\bullet$ and $\tilde \bullet$ are called equivalent if there exits a formal series $T = \id +\sum_{k=1}^\infty T_k \lambda^k$ of differential operators on $P$ such that 
	\begin{equation}
		T(a \bullet f) = a \mathbin{\tilde \bullet} T(f)
	\end{equation}
	for all $a \in \CCinf(M)$ and $f \in \CCinf(P)$.
\end{defn}
  
Since $T$ as above is always invertible, given a module deformation $\bullet$, one can define an equivalent module by $a \tilbullet f= T^{-1}(a \bullet T(f))$. If the module is fiber preserving and $T$ satisfies  $T(\pr^*a) =0$ for all $a \in \CCinf(M)$ the new module $\tilde \bullet$ will also be fiber preserving.

The bidifferential operators $L_k$ can also be considered as elements of $\DO(\CCinf(M),\allowbreak \DO(P))$ by considering the operators $a \mapsto L_k(a,\cdot)$. So it is possible to find the obstruction to an order by order construction of a module structure in the differential Hochschild cohomology $\HH^2_\diff(\CCinf(M),\DO(P))$. This goes back to \citep{gerstenhaber1}.

\begin{lemma}[{\citep[Propostion 2.4.3]{weisphd}}] \label{th:mho}
	Assume that $L^{(r)} = \sum_{k=0}^r \lambda^k L_k$ is a $(\CCinf(M),\star)$-left module structure up to order $\lambda^k$, with $a \star b = \sum_{k=0}^\infty \lambda^k C_k(a,b)$, then $L^{(r+1)} = L^{(r)} + \lambda^{r+1} L_{r+1}$ is a module structure up to order $k+1$ if 
	\begin{equation}
		\delta L_{r+1} = R_r,
	\end{equation}
	where $\delta$ is the Hochschild differential of $\HC^\bullet(\CCinf(M),\DO(P))$ and $R_r$ is given by 
	\begin{equation}
		R_r(a,b) = \sum_{k=0}^r L_k(C_{r+1-k}(a,b),\cdot) - \sum_{k=1}^r L_k (b,L_{r+1-k}(a,\cdot)).
	\end{equation}
	Also $\delta R_r =0$, whence the obstruction for an order by order construction of a module structure is $[R_r] \in \HH^2_\diff(\CCinf(M),\DO(P))$.
\end{lemma}

Similarly to the above lemma also the obstruction for the  construction of an equivalence order by order lies in a certain Hochschild cohomology.

\begin{lemma}[{\citep[Lemma 2.2]{art}}] \label{th:eho}
	Assume that $T^{(r)}  = \id + \lambda T_1 + \cdots + \lambda^r T_r$ is an equivalence between two left module structures $\bullet$ and $\tilbullet$ with differential operators $T_k$. Then the condition for $T^{(r+1)} = T^{(r)} + \lambda^{r+1} T_{r+1}$ to be an equivalence up to order $r+1$ is given by 
	\begin{equation}
		\delta T_{r+1} = E_r
	\end{equation}
	where $E_r(a)(f) = \sum_{s=0}^r  (L_{r+1-s}(a, T_s(f)) - T_s(L_{r+1-s}(a,f)) $
	Moreover $\delta E_r =0$ so the obstruction for an order by order construction lies in $\HH^1_\diff(\CCinf(M),\DO(P))$ for any order.
\end{lemma}

In fact the proofs are completely algebraic so the hold for any algebra and module.
Concerning the existence and equivalence, is was shown in \citep[Theorem 1.5]{art} that:
\begin{theorem}
	Given a fibered manifold $P \xrightarrow{\pr} M$ and a star product on $M$ there exists always a (fiber preserving) module deformation, which is unique up to equivalence.
\end{theorem}  

This follows also from \cref{th:hkrd}  using the previous statements.

\subsection{Bimodule deformations}\label{ch:bim}

We now come to the study of bimodule deformations of a fibered manifold $P \xrightarrow{\pr} M$.

\begin{defn}[Bimodule deformation]
	A bimodule deformation of a surjective submersion is a left and right module deformation, $\bullet$ and $\prbullet$ resp., such that 
	\begin{equation}
		(a \bullet f) \prbullet b = a \bullet (f \prbullet b)
	\end{equation}
	for all $a,b\in \CCinf(M)$ and $f\in \CCinf(P)$, i.e. $\CCinf(P)\ph$ becomes a $(\CCinf(M)\ph,\star)$-bimodule.
	It is called fiber preserving if both module structures are fiber preserving. 
\end{defn}
 
We will call both module structures $\bullet$ in the following, because from the context it is clear which one we mean.

Also for the case of bimodules it is possible to define a notion of equivalence:

\begin{defn}
	Two bimodule deformations $\bullet$ and $\tilbullet$ are called equivalent if there exists a 
	formal power series $T= \id + \sum_{k=1}^\infty T_k \lambda^k$ of differential operators on $P$ such that 
	\begin{align}
		T(a \bullet f) & = a \tilbullet T(f) \\
		T(f \bullet a) & = T(f) \tilbullet a 
	\end{align}
	In this case $T$ is called the bimodule equivalence.
\end{defn}
 
Note that $T$ is a left and a right module  equivalence.
  
A simple calculation gives the following
\begin{lemma}
	Given a bimodule deformation $(\bullet,\prbullet)$ and $T$ as in the above definition $(\tilbullet,\tilbullet')$ given by 
	\begin{align}
		a \tilbullet f = T^{-1}(a \bullet T(f))  \\
		f \tilbullet' a = T^{-1}(T(f) \bullet a) 
	\end{align}
	is an equivalent bimodule deformation.
\end{lemma}

In the definition of a bimodule deformation one can also consider the case, where the star product that acts from the left is different from the one that acts from the right. The following proposition shows that in nice situations this is not the case 

\begin{prop}
	Given a bimodule $(\bullet,\prbullet)$ over $\star$ and $\star'$, the two Poisson brackets are the same.
	If the bimodule is fiber preserving we even have $\star =  \star'$.
\end{prop}
\begin{proof}
	Since all left and right modules are equivalent and there always exists a fiber preserving one, we can assume that $\prbullet$ is fiber preserving, i.e. $R_1(\pr^*a,b) = \pr^*  C'_1(a,b)$, because we can use this right module equivalence as a bimodule equivalence.
	We can also find a left module equivalence $T = \id + T_1 \lambda + \O(\lambda^2)$, which would make the left module fiber preserving. This means there exist a $T_1$ such that the following equation holds:
	$$L_1(a,\pr^* b) =\pr^* C_1(a,b) + a T_1(\pr^* b) - T_1(\pr^* ab).$$
				  
	From the bimodule condition $ a \bullet (f \bullet b) = ( a \bullet f) \bullet b$  in first order we get 
	$$\pr^* b L_1(a,f) + R_1(\pr^* a f,b) - \pr^* a R_1(f,b) - L_1(a,\pr^* b f) =0$$
	Inserting $L_1$ and $R_1$ as above and setting $f=1$ gives:
	\begin{align*}
		\pr^* b C'_1(a,1) + \pr^*(ba) T_1(1) - b T_1(\pr^*a) -\pr^* C'_1(a,b) &    \\
		+ a T_1(\pr^* b ) -T_1(\pr^*ab) + \pr^* C_1(a,b)                      & =0 
	\end{align*}
	Exchanging $a$ and $b$  then subtracting the two equations gives 
	\begin{equation*}
		C_1(a,b) - C'_1(a,b) - ( C_1(b,a) - C'_1(b,a)) =0,
	\end{equation*}
	as we wanted.
	\\
	The second statement follows from  $a \bullet (1 \prbullet b) = a \bullet \pr^* b = \pr^* (a \star b)$ and
	$ (a \bullet 1) \prbullet b = \pr^*a \prbullet b =\pr^*( a \star' b)$. 
\end{proof}

In the last section we showed that an order by order construction of a module is equivalent to solving equations in a certain Hochschild cohomology. The same can be done for a bimodule deformation.
To see this, one uses the well known fact that an $\A$-bimodule is equivalent to an $\A^e = \A \otimes \A^\opp$-module. With this one gets that the right Hochschild cohomology to consider is $\HH^\bullet(\A^e,\DO(P))$.

We now want to define a semi-classical limit of an bimodule deformation, which in some sense generalizes the fact that the semiclassical limit of a star product is a Poisson bracket.

\begin{defn}
	Given a surjective submersion $P\rightarrow M$ with a bimodule structure
	 $(\bullet ,\prbullet)$, with $a \bullet f = \sum_{k=0}^\infty L_k(a,f) \lambda^k$ and $f \prbullet a = \sum_{k=0}^\infty R_k(f,a)$,
		we can define the semi-Poisson bracket (sP-bracket)
		 $\FP{\cdot,\cdot} : \CCinf(M) \times \CCinf(P) \rightarrow \CCinf(P)$
		  by	
	\begin{equation}
		\FP{a,f} := \frac{\I}{2}(L_1(a,f) - R_1(f,a)).
	\end{equation}
\end{defn}
The factor $\frac{\I}{2}$ assures compatibility with the Poisson bracket.

\begin{remark}
	One can make the same definition if $\A$ is an arbitrary commutative algebra and $\M$ is a symmetric $\A$-bimodule. Also the following proposition remains true in this context.
\end{remark}

\begin{prop}\label{fpprop}
	The sP-bracket satisfies
	\begin{enumerate}[i)]
		\item $\FP{ab,f} = \pr^* a\FP{b,f} + \pr^* b \FP{a,f}$
		\item $\FP{a, \pr^* b f} = \pr^* \p{a,b} f  + \pr^* b \FP{a,f}$
		\item $\FP{a, \FP{b,f}} - \FP{b, \FP{a,f}} -\FP{ \p{a,b},f} =0$,
	\end{enumerate}
	for all $a,b \in \CCinf(M)$ and $f \in \CCinf(P)$.
	So especially the sP-bracket is a derivation in the first argument.\\
	If the bimodule is fiber-preserving, we also have $\FP{a,\pr^* b} = \pr^* \p{a,b}$.
\end{prop}
\begin{proof}
	Similarly how one can get the properties of a Poisson bracket from the associativity of the  star product, one also gets these properties of the sP-bracket from the compatibility of left and right module and the star product.
	\begin{enumerate}[ {ad} i)]
		\item We consider the equation 
		      \begin{equation*}
		      	(a \star b) \bullet f - f \bullet (a \star b) = a \bullet (b \bullet f) - a \bullet ( f \bullet b)  + (a \bullet f) \bullet b - (f \bullet a) \bullet b
		      \end{equation*}
		      Evaluating in order $\lambda$ gives
		      \begin{align*}
		      	\pr^* C_1(a,b)  f & + L_1(ab,f) - R_1(f,ab) - \pr^* C_1(a,b) f                                    \\
		      	                  & = L_1(a, \pr^* b f) + \pr^* a L_1(b,f)) - \pr^* a R_1(f,b) - L_1(a,\pr^* b f) \\
		      	                  & + \pr^* b L_1(a,f)  + R_1(\pr^* a f,b) - \pr^* b R_1(f,a) - R_1(\pr^* a f ,b) 
		      \end{align*}
		      Some of the terms cancel and the remaining give the desired equation.
		\item $a \bullet (b \bullet f) -  (b \bullet f) \bullet a = b \bullet (a \bullet f - f \bullet a) + (a \star b - b\star a) \bullet f$ in first order gives the result.
		\item Using $ \FK{a, f}  =  a \bullet f - f \bullet a$ this equation is the second order term of 
		      \begin{equation}
		      	\FK{a, \FK{b,f}} = \FK{ [a,b]_\star , f} + \FK{b,\FK{af}}.
		      \end{equation}
		      This does not evolve the second order term of the module, since the zeroth order term of $\FK{\cdot,\cdot}$ is zero.
	\end{enumerate}
\end{proof}

A bracket which satisfies the properties given in the previous proposition is sometimes called a Poisson module. Note these are completely algebraic.
In the following we will call a bracket which satisfies these properties a semi-Poisson bracket.

\begin{prop}
	The sP-bracket of a bimodule deformation is invariant under bimodule equivalence transformations. So let $(\bullet,\prbullet)$ and $(\tilbullet,\tilbullet')$ be two equivalent bimodules and $\FP{\cdot,\cdot}$ and $\FP{\cdot,\cdot}'$  resp. be the corresponding sP-brackets then we have
	\begin{equation}
		\FP{a,f} = \FP{a,f}' \textrm{ for all } a\in \CCinf(M), f \in \CCinf(P).
	\end{equation}
\end{prop}
\begin{proof}
	Let $T  \in \DO(P) \ph$ be the bimodule equivalence, then one has
	\begin{equation*}
		\tilde L_1 (a,f) = L_1(a,f) + a T_1(f) -  T_1(af)
	\end{equation*}
	and
	\begin{equation*}
		\tilde R_1 (f,a) = L_1(f,a) + a T_1(f) -  T_1(af).
	\end{equation*}
	Subtracting these two one obtains
	\begin{equation*}
		\tilde L_1(a,f) - \tilde R_1(f,a)  = L_1(a,f) -R_1(f,a).
	\end{equation*}
\end{proof}

This means for example that two bimodule deformations with different sP-brackets cannot be equivalent.

\begin{defn}
	We will call a sP-bracket fiber preserving if $\FP{a,\pr^* b }= \pr^* \p{a,b}$, this is equivalent to $\FP{a,1} =0$, since $\FP{a,\pr^* b 1} = \pr^* \p{a,b} 1 + \pr^* b \FP{a,1}$. 
	We will call a sP-bracket natural if $\FP{a,fg}= \FP{a,f}g + \FP{a,g}f$. So a natural sP-bracket is also fiber-preserving.   
\end{defn}

\begin{prop}
	If a bimodule deformation is  fiber preserving  so is  the corresponding sP-bracket.
	If  it is fiber preserving and natural the corresponding sP-bracket is natural.
\end{prop}

We recall the definition of a Hamiltonian vector field. Let $M$ be a Poisson manifold and $a \in \CCinf(M)$ then we define the Hamiltonian vector field $X_a \in \CCinf(M)$ by  $X_a(b) = \p{a,b}$ for $b \in \CCinf(M)$. Note that sometimes a different sign is chosen.

\begin{prop}\label{lift}
	Given a natural sP-bracket on $P \xrightarrow{\pr} M$, where the corresponding Poisson bracket is symplectic, we get a horizontal lift, which is given on Hamiltonian vector fields by $X_a^h(f) = \FP{a,f}$ for $a\in \CCinf(M)$ and $f \in \CCinf(P)$, and thereby a connection on $P$.
\end{prop}
\begin{proof}
	Since $M$ is symplectic it is enough to specify the horizontal lift on Hamiltonian vector fields $X_a \in \X(M)$, since these span the tangent space at every point. For these we set $X_a^h(f) = \FP{a,f}$ for all $f \in \CCinf(P)$.
	This is well-defined because the sP-bracket is a derivation in the first argument so it only depends on the differential of $f$. Since the Poisson structure is symplectic this is uniquely determined by the vector field. 
	Because we assume $\FP{\cdot,\cdot}$ to be natural it is also a derivation in the second argument and so $X_a^h$ is really a vector field. 
	Finally, since $X_a^h(\pr^* b) =  \FP{a, \pr b} =\pr^* \p{a,b}  = \pr^* X_a(b)$, we get a horizontal lift.
\end{proof}

Now we come to a main result of this section:

\begin{theorem}\label{th:flat}
	Given a natural sP-bracket on $P \xrightarrow{\pr} M$, where the corresponding Poisson bracket is symplectic, the connection defined in \cref{lift} is flat.\\
	So given a fibered manifold $P$ over a symplectic manifold $M$ a bimodule deformation with natural sP-bracket can only exists if $P$ admits a flat connection.
\end{theorem}
\begin{proof}
	Since the manifold is assumed to be symplectic it suffices to compute the curvature on Hamiltonian vector fields, for these we get with the Jacobi identity (\Cref{fpprop})
	\begin{align*}
		R(X_a,X_a)(f) & = [X_a^h ,X_b^h](f) - [X_a,X_b]^h(f)                    \\
		              & = X_a^h (X_b^h(f)) - X_b^h (X_a^h(f)) - X_{\p{a,b}} (f) \\ 
		              & =\FP{a,\FP{b,f}}-\FP{b,\FP{a,f}} - \FP{ \p{a,b} ,f} =0  
	\end{align*}
	for all $a,b \in \CCinf(M)$ and $f \in \CCinf(P)$.
\end{proof}

This can be generalized in some sense to the case where the sP-bracket is not natural. 

\begin{prop}\label{rm:lift}
	In the symplectic case  the sP-bracket can be used  to define a map $\DO(M) \to \DO(P)$.
\end{prop}
\begin{proof}
	Before giving the proof, we recall very briefly the constructing of the universal enveloping algebra of a Lie-Rinehart algebra as given in \citep{MR1058984}. 
 Let $(A,L)$ be a Lie-Rinehart algebra over $\K$. Then we define $A \odot U(L)= A \otimes U(L)$, where $U(L)$ denotes the universal enveloping algebra of $L$ considered as Lie algebra over $\K$, with the multiplication $ (a \otimes X)(b \otimes Y) = ab \otimes XY + a X(b) \otimes Y$ with $a,b \in A$ and $X,Y \in L$.
Then we denote by $I$ the ideal generated by $a \otimes X - 1 \otimes a X$ and $U(A,L) =\factor{A \odot L}{I}$ is the universal enveloping algebra of $(A,L)$. We recall that for a manifold $M$ we have $U(\CCinf(M),\X(M)) =\DO(M)$.  
 On Hamiltonian vector fields we can define a Lie algebra morphism   $\X(M) \to \DO_L(P)$, where the Lie bracket on $\DO(M)$ is given by the commutator, by  $X_a \mapsto \FP{a, \cdot}$. This can be extend to an algebra morphism $\phi:  U(\X(M)) \to \DO(M)$.  With this we define a map $\Phi:\CCinf(M) \odot U(\X(M)) \to \DO(P)$ by $a \otimes  D \mapsto \pr^* a D $. Since 
\begin{align*}
	\Phi(a \otimes X_c )\Phi( b \otimes X_d) (f)&= \pr^* a \FP{ c,\pr^* b \FP{d,f}} = \pr^*( a b)\FP{c,\FP{d,f}} + \pr^*(a \p{c,b}) \FP{d,f} \\
	&=	\Phi(ab \otimes X_c X_d + a  X_c(b)\otimes X_d  ) = \Phi((a\otimes X_c)(b \otimes X_d))
\end{align*}
for $a,b,c,d \in \CCinf(M)$ it is an algebra morphism. It is also clear that it vanishes on $I$ and we get an induced map from $\DO(M)= U(\CCinf(M),\X(M)) \to  \DO(P)$.
\end{proof}

In the non-symplectic case we get a horizontal lift over the symplectic leaves of the Poisson manifold. 
This condition is clearly not enough to get a bimodule deformation. Consider for example a symplectic star product $\star$ and formally replace $\lambda$ by $\lambda^2$ to get $\tilde \star$ then the Poisson tensor $\tilde \pi =0$,  so the condition on the symplectic leaves is empty, but we can only find a bimodule for $\tilde \star$ if there is one for $\star$.  
  
We do not get a horizontal lift in the  non-symplectic case due to two reasons:
\begin{itemize}
	\item When the Poisson tensor is degenerate, the Hamiltonian vector fields do not span the tangent space, so it is not possible to lift every vector.
	\item The horizontal lift would be ill defined, because it can happen that $X_a =X_b$ for $a,b \in \CCinf(M)$ with $\d a \neq \d b $.
\end{itemize}

The obstruction we find is only in first order in $\lambda$ and in the general Poisson case it is the existence of a sP-bracket. This is non trivial as we have seen. In the symplectic the existence of a sP-bracket is enough to get a deformation. 

\begin{theorem}
	Let $(M,\star)$ be a manifold with a symplectic star product  and $P \to M$ a fibered manifold. Given a sP-bracket on $P$, there exists a bimodule deformation.
\end{theorem}
\begin{proof}
	The star product consists of differential operators $C_i$ on $M$, which can be lifted by using \cref{rm:lift}. So we set $L_i(a,f) := C_i(a,\cdot)^h(f)$ and $R_i(f,a) := C_i(\cdot,a)^h(f)$. One has to check that this in fact defines a bimodule. This follows since 
	$C_i(a,L_j(b,f)) = C_i(a, C_j(b, \cdot ))^h(f)$, which show that it is a left module and similar for the other conditions.
\end{proof}

One can also consider the case were only a Poisson structure on $M$ is given and not  a star product. Then the next question would be, if given a sP-bracket  it is possible to get a a star product on $M$ and bimodule structure on $\CCinf(P)$. Further it would interesting to classify them  up to equivalence, similarly to Kontsevich's formality theorem.

\subsection{Deformation of principal fiber bundles}

In this section we want to consider the deformation of principal fiber bundles. We denote the structure group by $G$. We then have an induced left action of $G$ on the functions on $P$, given by 
$(g \triangleright f)(p) = f( p \cdot g)$, where $\cdot$ denotes the principal right action.

\begin{defn}
	A module deformation of a principal fiber bundle with structure group $G$ is called a deformation of a principal fiber bundle if 
	\begin{equation}
		g \triangleright (a \bullet f) = a \bullet (g \triangleright f)
	\end{equation}
	for all $g \in G$ and similarly for a bimodule.
\end{defn}

\begin{prop}
	For a bimodule deformation of a principal bundle we have for the sP-bracket
	\begin{equation}
		g  \triangleright \FP{a,f} = \FP{ a, g \triangleright f}.
	\end{equation}
\end{prop}
\begin{proof}
	Since $g \triangleright (a \bullet f) = a \bullet (g \triangleright f)$ we get $g \triangleright L_1(a,f)= L_1(a,g \triangleright f)$
	and similarly for $R_1$, with this
	\begin{equation*}
		g  \triangleright \FP{a,f} = g  \triangleright (L_1(a,f) -R_1(f,a)) =
		L_1(a,g  \triangleright f) - R_1(f, g  \triangleright f) = \FP{a, g  \triangleright f}.
	\end{equation*}
\end{proof}

\begin{prop}
	In the case of a principal bundle bimodule deformation the connection of \cref{lift} is a principal connection.
\end{prop}
\begin{proof}
	For a Hamiltonian vector field $X_a$ we have 
	\begin{align}
		g  \triangleright X^h_a(f) = g \triangleright  \FP{a,f} = X^h_a(g  \triangleright f), 
	\end{align}
	which shows that the horizontal lift and so the connection is compatible with the principal fiber bundle.
\end{proof}

If one has a deformation of a principal bundle one can also define a deformation of associated vector bundles, for the module case this is done in  \cite[Sec. 6]{art}. Here we proceed similarly for the bimodule and algebra case.
 
So let us consider an associated vector bundle $E=P[V,\rho]$ over a principal $G$ bundle $P$ with typical fiber a finite dimensional vector space $V$ and $\rho: G \rightarrow \GL(V)$ a representation. For details of the definition see \citep[Section 18.7]{michor}.
It is well known that $\CCinf(P,V)^G \cong \Gamma^\infty(E)$. We denote this isomorphism by $\hat \cdot$ and its inverse by $\check \cdot$. 
 
\begin{defn}
	Given an associated vector bundle $E=P[V,\rho]$ and a  bimodule deformation of $P$ as a principal bundle, we define a bimodule structure on $\Gamma^\infty(E)$ by 
	\begin{equation}
		a \bullet s = (a \bullet \hat s)^\vee 
	\end{equation}
	and similarly for the right module structure.
\end{defn}
\begin{proof}
	We need to show that this is well defined.  For this we have to show  that $a \bullet \hat s$ is again $G$ equivariant. We have for any $g \in G$
	\begin{equation}
		g \triangleright (a \bullet \hat s) = a \bullet ( g\triangleright \hat s) = a \bullet \rho(g) \hat s =
		\rho(g) a \bullet \hat s.
	\end{equation}
	What is what we wanted.
\end{proof}

\begin{remark}
	If $V$ is also an algebra and given a principal subalgebra deformation of $P$, one can also define an  $\CCinf(M)$ algebra   structure on $\Gamma^\infty(E)$.
\end{remark}
\begin{proof}
	For $f,g \in \CCinf(P,V)^G \cong \Gamma^\infty(E)$ and $\{e_i\}$ a basis of $V$ one defines 
	\begin{equation}\label{sa1}
		f \star g = (f^i e_i) \star g^j (e_j)  = (f^i \star g^j) (e_i \cdot e_j) 
	\end{equation}
	where $\cdot$ is the undeformed product of the algebra $V$.
	This is obviously independent of the choice of the basis. Using the fact that $g \triangleright (f \star_P h) =  (g \triangleright f) \star_P (g \triangleright h)$ for $g \in G$ and $f,h \in \CCinf(M)$, one gets that the product in \eqref{sa1} is again $G$-equivariant.
\end{proof}

An interesting example of this is the frame bundle of a manifold, because if we can deform this as an algebra we can also deform the associated vector bundles like the tangent or cotangent bundle and  higher tensor bundles like the exterior algebra. Deforming a single bundle of this as an algebra is straight forward using the above remark. More care has to be taken if the relations between these, e.g. that the tangent bundle is the dual of the cotangent bundle, should be preserved.
One also should note that if the Poisson structure on $M$ is symplectic, these deformations can only exist if the frame bundle is trivial, i.e. the manifold is parallelizable. But even in this case we do not see a straightforward way of deforming for example the de-Rham differential. But also in other approaches to deforming the exterior algebra this only works for specific cases, for example using a Drinfeld twist, see e.g. \citep{wess}.

\subsection{Equivalence of bimodules}\label{ch:nonequi}

In this section we want to show that given a trivial fiber bundle there are in general infinitely many nonequivalent bimodule deformations. For this we first describe a way of constructing a new bimodule structure out of a given one.
 
We denote $\mu_\star(a \otimes b) = a \star b$, $l_\star(a \otimes f) = a \bullet f$ and $r_\star(f \otimes a) = f \bullet a$.
We can define a new left module structure by $l'_\star = l_\star \E^{\lambda Q}$. The exponential is well defined since in any order in $\lambda$ there are only finitely many terms.
Here 
\begin{equation}
	Q= \sum_i E_i \otimes D_i
\end{equation}
where 
$E_i$  is a derivation of the star product, this means $E_i(a \star b) = E_i(a) \star b + a \star E_i(b)$ or 
\begin{align} 
	E_i \circ \mu_\star = \mu_\star (E_i \otimes \id + \id \otimes E_i) 
\end{align}
and $D_i$ is a homomorphism  of the bimodule, so $D(a \bullet f \bullet b) = a \bullet D(f) \bullet b$ or  
\begin{align}
	D \circ l_\star = l_\star (\id \otimes D) & \text { or } D(a \bullet f) = a \bullet D(f) \text{ and } \\
	D \circ r_\star = r_\star (D \otimes \id) & \text { or } D(f \bullet a) = D(f) \bullet a              
\end{align}
We also have to assume that all the $D_i$ commute among each other and also all the $E_i$, i.e. $D_i \circ D_j - D_j \circ D_i$ for all $i,j$ and the same for the $E_i$. 
We also use
\begin{equation}
	Q_{13} = E_i \otimes \id \otimes D_i,
\end{equation}
where a sum over $i$ is to be understood as in the following computations.
 
We show that $l'_\star$ is again a left module and together with $r_\star$ a bimodule.
For this we first compute 
\begin{align}
	\begin{split}\label{ec1}
	Q \circ (\id \otimes l_\star) & = (E_i \otimes D_i)(\id \otimes l_\star) = E_i \otimes D_i l_\star                            \\
	                              & = E_i \circ l_\star  (\id  \otimes D_i) = (\id \otimes l_\star )(E_i \otimes \id \otimes D_i) \\
																& = (\id \otimes l_\star ) Q_{13} 
	\end{split}                                                              
\end{align}
and also 
\begin{align}\label{ec2}
	\begin{split}
	Q \circ (\mu_\star \otimes \id) & = (E_i \otimes D_i) \circ (\mu_\star \otimes \id)                                                                      \\
																	& =( \mu_\star \circ  (E_i \otimes \id + \id \otimes E_i)) \otimes D_i = (\mu_\star \otimes \id) \circ (Q_{13} + Q_{23}) 
	\end{split}
\end{align}
and
\begin{align}\label{ec3}
	Q \circ (\id \otimes r_\star) = (E_i \otimes D_i) (\id \otimes r_\star) =  E_i \otimes D_i \circ  r_\star  = E_i \otimes r_\star (\id \otimes D_i) =(\id \otimes r_\star )\circ Q_{12}. 
\end{align}
 
Next we show $(a \star b) \prbullet f = a \prbullet ( b \prbullet f)$. Using \eqref{ec2} we compute 
\begin{align*}
	l'_\star  (\mu_\star(a \otimes b) \otimes f)        
	  & =     l_\star \E^{\lambda Q} (\mu_\star(a \otimes b) \otimes f)           \\
	  & =	l_\star \E^{\lambda Q} (\mu_\star \otimes \id)( a\otimes b \otimes f) = 
	l_\star (\mu_\star \otimes \id) \E^{\lambda(Q_{13} + Q_{23})}
\end{align*}
and using \eqref{ec1} 
\begin{align*}
	l'_\star(a \otimes l'_\star( b\otimes f)) & =                                                                                                 
	l_\star \E^{\lambda Q} ( a \otimes l_\star \E^{\lambda Q} (b \otimes f)) \\ 
	                                          & =  l_\star \E^{\lambda Q} \circ (\id \otimes l_\star) \E^{\lambda Q_{23}} (a \otimes b \otimes f) \\
	                                          & =  l_\star (\id \otimes l_\star ) \E^{\lambda (Q_{13} + Q_{23})} (a \otimes b \otimes f)    .     
\end{align*}
In the last step we used the fact that the $D_i$ commute to get $ \E^{\lambda (Q_{13} + Q_{23})}= \E^{\lambda Q_{13}} \E^{\lambda Q_{23}}$.
 
Comparing these two and using that fact that $l_\star$ is a left-module  gives $(a \star b) \prbullet f = a \prbullet ( b \prbullet f)$ as wanted.
 
Next we want to show that it is also a bimodule, so we compute $ ( a \prbullet f) \bullet b$ and $ a  \prbullet ( f \bullet b)$:
\begin{align*}
	r_\star (l_\star \E^{\lambda Q} (a \otimes f) \otimes b ) = r_\star \circ (l_\star \otimes \id) \circ \E^{\lambda Q_{12}} (a\otimes f \otimes b) 
\end{align*} 
\begin{align*}
	l_\star \E^{\lambda Q} (a\otimes r_\star (f \otimes b)                       
	  & =  	l_\star\E^{\lambda Q_{23}} \circ (\id \otimes r_\star) ( a \otimes f \otimes b) \\
	  & = l_\star \circ (\id \otimes r_\star) \E^{\lambda Q_{12}} ( a \otimes f \otimes b)  
\end{align*}
The two sides agree, since $l_\star$ and $r_\star$ form a bimodule, so we get in fact a bimodule.
 
This shows the following proposition:
\begin{prop}
	Let $(\bullet,\prbullet)$  be a bimodule deformation of $P \xrightarrow {\pr} M$, $n \in \N$, $E_i$ for $i=1, \dots,n$ be a derivation of the star product and $D_i$ for $i=1,\dots n$ a bimodule homomorphism, such that all the $E_i$ commute and also the $D_i$. Then $l_\star \E^{\lambda\sum_i E_i \otimes D_i }$, where $l_\star$ denotes the original left module structure, is again a bimodule structure with the same right module structure.
				 
	The modified bimodule has the sP-bracket
	\begin{equation}
		\FP{a,f}' = \FP{a,f} + \pr^* E_i(a) D_i(f).
	\end{equation}
\end{prop}
 
The statement for the sP-bracket follows directly  from 
\begin{equation}
	l'_\star = l_\star \circ \E^{\lambda Q} = L_0 + \lambda (E_i \otimes D_i) + \lambda L_1 + \O(\lambda^2)
\end{equation}
and the definition of the sP-bracket.
 
With this construction it is at least in the trivial case  possible to construct lots of different, i.e. nonequivalent bimodule structures,
because there always exist derivations of a star product, e.g. the quasi-inner ones, and any  vertical differential operator, whose coefficients are also independent of $M$, gives a bimodule homomorphism for the bimodule described in the first part of \cref{ch:examples}, which gives the following corollary:
 
\begin{corollary}
	Let $M\times F \rightarrow M$ be a trivial fiber bundle and $\star$ a star product on $M$. Then the there are infinitely many nonequivalent bimodule deformations.
\end{corollary}

Here it can also be seen that even two bimodule deformations having the same sP-bracket are not equivalent. Take the trivial Poisson bracket with the trivial star product and also the trivial bimodule deformation. Then take $\tilde l_\star = l_\star \circ \E^{\lambda^2 Q}$ with $Q$ as above  this does not change the sP-bracket nor the right module but in general changes the left module structure. In contrast in this case every bimodule equivalence which changes the left also changes the right module structure.

\subsection{Examples}\label{ch:examples}

First of all there is the trivial example. So let $P= M \times G$ with manifolds $M,G$ and $\star$ a star product on $M$.
Then we can choose the trivial connection and lift the differential operators in $\star$ with this and define $a \bullet f= \sum_{r=1}^\infty C^h_r(a,f)$
This means that for $f: (x,g) \mapsto f(x,g)$ the operators only act on $x$. This clearly gives a bimodule.

Actually its enough to have a flat connection on $P \rightarrow M$. Then lifting the differential operators $C_k(a,\cdot)$ and $C_k(\cdot,a)$ in the star product to differential operators on $N$ gives a left- resp. right module  structure.

\begin{prop}
	Consider  a fibered manifold $P \rightarrow M$ and commuting vector fields 
	$X_1 \ldots X_k$ on $M$ and a horizontal lift such that also the $X^h_i$ commute. 
	Then we can define for any constant matrix $A=(a_{ij})$ a star product on $M$ by
	\begin{equation}
		a \star b= \mu(\E^{a^{ij} X_i \otimes X_j} a \otimes b)
	\end{equation}
	for $a,b \in \CCinf(M)$ and analogue on $P$ by
	\begin{equation}
		a \star_P b= \mu(\E^{a^{ij} X^h_i \otimes X^h_j} f \otimes g)
	\end{equation}
	for $f,g \in \CCinf(P)$. 
				
	Then $(\CCinf(M),\star)$  is a subalgebra of  $(\CCinf(P),\star_P)$, so we have a subalgebra deformation as described in the following section, and we also get a bimodule.
\end{prop}
\begin{proof}
	First we note that $\star$ and $\star_P$ really define two star products. See e.g. \citep[Sect.6.2.4]{defquan} for a proof of this. So we only need to show that we get a subalgebra.
	This follows form 
	\begin{align*}
		\pr^* a \star_P \pr^*b & = \mu(\E^{a^{ij} X^h_i \otimes X^h_j} \pr^* a \otimes \pr^b) \\
		                       & = \mu( \pr^* \E^{a^{ij} X_i \otimes X_j} a \otimes b)        
		= \pr^* (a \star b),
	\end{align*}
	for all $a,b \in \CCinf(M)$, where we used $X^h \pr^a = \pr^* X(a)$.
	The bimodule structure is given by $a \bullet f \bullet b = \pr^* a \star_P f \star_P \pr^* b$.
\end{proof}

\subsection{Subalgebra deformation}

We briefly want to give some remarks  on the deformation of $\CCinf(M)$ as a subalgebra of $\CCinf(P)$ for a fibered manifold $P \xrightarrow{\pr} M$. This has  already been considered in \citep{math/0403334}, but we here want to relate it to bimodule deformations. We also do not assume a fixed given Poisson bracket on $P$.

\begin{defn}
	Given a fibered  manifold $P \xrightarrow{\pr} M$ and a star product $\star $	 on $M$, we call an algebra $(\CCinf(P)\ph,\star_P)$  a deformation as subalgebra of $P$ if 
	\begin{equation}\label{defsubalg}
		(\pr^* a) \star_P (\pr^* b) = \pr^* (a \star b) 
	\end{equation} 
	for all $a,b \in \CCinf(M)$.
\end{defn}

\begin{remark}
	Of course given a subalgebra deformation, one  also gets a bimodule deformation by defining 
	$a \bullet f = \pr^* a \star_P f$ and $f \bullet b = f \star_P \pr^* a$. This bimodule deformation is always fiber preserving.
\end{remark}

\begin{defn}
	Given a principal $G$ bundle $P \xrightarrow{\pr} M$ we call a subalgebra deformation $\star_P$ of $P$ a principal subalgebra if 
	\begin{equation}
		(g \triangleright f)  \star_P   (g \triangleright h)  = g \triangleright (f  \star_P h)
	\end{equation}
	for all $f,h \in \CCinf(P)$ and $g \in G$.
\end{defn}

\begin{prop} \label{th:sap}
	Given a subalgebra deformation of a fibered manifold $P$ and   $\p{\cdot,\cdot}_P$ the corresponding Poisson bracket, for $a,b \in \CCinf(M)$ we have 
	\begin{equation} \label{subpois}
		\pr^* \p{a,b} = \p{\pr^* a, \pr^* b}_P
	\end{equation}
\end{prop}
\begin{proof}
	This is a simple consequence from \eqref{defsubalg} in first order in $\lambda$.
\end{proof}

\begin{theorem} \label{th:sp}
	Given a fibered manifold $P \xrightarrow {\pr} M$ and a Poisson bracket $\p{\cdot,\cdot}_P$ on $P$ and $\p{\cdot,\cdot}$ on $M$, which is symplectic, satisfying $\pr^* \p{a,b} = \p{\pr^* a, \pr^* b}$, i.e. $\CCinf(M)$ is a Poisson subalgebra of $\CCinf(P)$, we
	get a horizontal lift, which is flat.
\end{theorem}
\begin{proof}
	Since $M$ is symplectic it is enough to consider Hamiltonian vector fields. 
	We define  
	\begin{equation}
		X_a^h(f) = \{\pr^* a, f\}_P \text{ for } f \in \CCinf(P).
	\end{equation}
	Since the Poisson bracket on $P$ is a derivation in the second arguments $X_a^h$ is really a vector field.
	The lift is well defined similarly to \cref{lift}.
	From  $X_a^h(\pr^* b) = \{\pr^* a,\pr^* b\}' = \pr^* \{a,b\} = \pr^* X_a(b)$ we see that $X_a^h$ is a horizontal lift of $X_a$.  
	For the curvature one finds
	\begin{align*}
		R(X_a,X_b)(f) & = [X_a^h ,X_b^h](f) - [X_a,X_b]^h(f)                                                       \\
		              & = X_a^h(X_b^h(f)) - X_a^h(X_b^h(f)) - X^h_{\{a,b\}}                                        \\
		              & = \{\pr^*a, \{\pr^*b ,f\}'\}' - \{\pr^*a,\{\pr^*b,f\}'\}' - \{\pr \{a,b\},f\}'             \\
		              & = \{\pr^*a, \{\pr^*b ,f\}'\}' + \{\pr^*a,\{f,\pr^*b\}'\}' + \{f,\{\pr^* a,\pr^* b\}'\}' =0 
	\end{align*}
\end{proof}

\begin{corollary}
	A subalgebra deformation of a fibered manifold $P \xrightarrow{\pr} M$ can only exist  if $P$ admits a flat lift.
\end{corollary}
\begin{proof}
	This is obvious from \cref{th:sp} and \cref{th:sap}
\end{proof}
Of course this also follows from \cref{th:flat} since a subalgebra gives rise to a fiber-preserving bimodule.

On the other hand if we have a flat lift, for example if one already has a subalgebra deformation or a $\pr^*$ related Poisson bracket on $P$, one can use this horizontal lift to lift the differential operators $C_k$ in the star product to get a star product on $P$, which gives a subalgebra deformation. The Poisson bracket on $P$ in this case is the lift of the Poisson bracket on $M$. 
But it turns out that this lifted Poisson bracket is in general different from the original one on $P$, since the lifted one can contain a term with both vector fields vertical. For example the Poisson structure on $M$ could be zero, but the one on $P$ is only vertical but nonzero.

\section{Generalization of the HKR theorem}

In this section we want to compute  the differential Hochschild cohomology of $\CCinf(P)$ and $\DO(P)$ as $\CCinf(M)$-bimodules. For this we first need some more technical constructions, for which we follow \citep{art}.

\subsection{Hochschild cohomology}\label{ch:hc}

Let $\M$ be a bimodule over an algebra $\A$ and 
\begin{equation}
	\HC^k (\A,\M)= \Hom(\A^{\otimes k},\M) \cong \Hom(\underbrace{\A,\ldots, \A}_{k \textrm{ times}},\M),
\end{equation}
where the isomorphisms follows from the universal property of the tensor product.
Here $\Hom(\A, \ldots ,\A;\M)$ means the multilinear maps  from $\A$ to $\M$ as vector spaces. 

Then we can define a differential on the complex $\HC^\bullet(\A,\M)$ by
\begin{align}
	\delta_i^{n}f(a_1\otimes \ldots \otimes a_{n+1}) = 
	\begin{cases} 
	a_1\cdot f(a_2\otimes \ldots \otimes a_{n+1})                        & \text{for }i=0   \\
	f(a_1\otimes \ldots \otimes a_ia_{i+1}\otimes \ldots\otimes a_{n+1}) & \text{for }0<i<n \\  
	f(a_1\otimes \ldots \otimes a_n)\cdot a_{n+1}                        & \text{for }i=n   
	\end{cases}
\end{align}
\begin{equation}
	\delta^n := \sum_{i=0}^n (-1)^i \delta^n_i
\end{equation}

One can compute that  $\delta^n \circ \delta^{n-1} =0$, so we can make the following definition:
\begin{defn}
	The $n$-th cohomology group of $\HC^n(\A,\M)$ can be  defined by
	\begin{equation}
		\HH^n(\A,\M) =\factor{ \ker(\delta^{n+1})}{\im(\delta^n)}
	\end{equation}
	This is the so called Hochschild cohomology of $\M$.
			
	We note that if $\A$ is commutative, $\HH^n(\A,\M)$ is again an $\A$-bimodule.
\end{defn}

Later we will consider the case $\A = \CCinf(M)$ for a manifold $M$ and $\M = \DO(N)$ or $\M = \CCinf(N)$ where we have a map $N \xrightarrow{\pr} M$.  The bimodule structure on  $\DO(N)$ is defined by
\begin{equation}
	(a \cdot D \cdot b)(f) =  \pr^* a D( \pr^* b f)
\end{equation}
and similarly for $\CCinf(P)$.

\begin{defn}[Differential Hochschild complex]
	Let $\A$ be a commutative algebra, then we define the differential Hochschild complex by 
	\begin{equation}
		\HC^\bullet_\diff(\A,\M) = \bigoplus_{k=0}^\infty \HC^k_\diff (\A,\M) \subset \HC^\bullet(\A,\M)
	\end{equation}
	with 
	\begin{equation}
		\HC^k_\diff(\A,\M) = \DO^\bullet(\underbrace{\A,\ldots,\A}_{k \text{ times}};\M).
	\end{equation}
	The corresponding Hochschild cohomology we denote by $\HH^\bullet_\diff(\A,\M)$.
\end{defn}

In the case of $\HC_\diff(\A,\DO(P))$ we slightly modify this and set 
\begin{equation}
	\HC_\diff^k(\A,\DO(P)) = \bigcup_{L \in \N^k_0} \bigcup_{l \in \N_0} \DO^L(\A,\dots,\A;\DO(P)).
\end{equation}

To see that this actually is a subcomplex one needs that $\M$ is a differential bimodule, to make sure that the differential restricts to the set of differential operators. The multiplication with an element of the algebra and the concatenation of differential operators is again a differential operator, so the only operation in the definition of the differential which needs not to be a differential operator is the right module multiplication.

\begin{defn} [Differential bimodule] \label{de:difbim}
	Let $\M$ be an $\A$-bimodule. Then $\M$ is called a differential bimodule if for all $a \in \A$ the map $\M \rightarrow \M: f \mapsto f \cdot a$ is a differential operator.
\end{defn}

The algebra we will use later will be $\CCinf(M)$ for a manifold $M$. This is a Fréchet algebra, with the usual  seminorms.
On $\DO$ we use the topology given by the local presentation. This is $D = \sum_I D^I \del_I$ for some multiindex $I$. We have seminorms for all $I$ given by the seminorms of $D^I$ considered as smooth functions. 

Since we are not interested in arbitrary homomorphisms but only in continuous ones  we also define

\begin{defn}[Continuous Hochschild complex]
	Let $\A$ be a commutative topological algebra and $\M$ a topological bimodule then we define  the continuous Hochschild complex by 
	\begin{equation}
		\HC^k_\cont(\A,\M) = \Hom_\cont(\A, \ldots, \A;\M) \subset \HC^k(\A,\M)
	\end{equation}
	where $\Hom_\cont$ denotes the space of all continuous homomorphism.
			
	Since $\delta$ maps continuous homomorphisms to continuous homomorphisms, it can be restricted to this subcomplex and we get the continuous Hochschild cohomology.
\end{defn}

Since in our situation every differential operator is continuous we have $\HC_\diff \subset \HC_\cont$.

\subsection{Bar complex}

We recall the definition of the bar complex adopted  to our situation following \citep{art,weisphd}.

We consider $\A =\CCinf(V)$ for an  convex open  subset $V$ of $\R^n$  and use $\A^e = \A \otimes \A^\opp$, which is an algebra for the obvious componentwise multiplication.
In our case of course $\A^\opp = \A$ but in general one needs $\A^\opp$. For the complexes we use, we actually need the completion of $\A^e$ in the projective topology of the tensor product, which we will denote by $\otimeshat$.

\begin{defn}
	We define the bar complex $X_\bullet$ as 
	\begin{align}
		X_0 = \A^e = \CCinf(V \times V)     \\
		X_k = \CCinf(V \times V^k \times V) 
	\end{align}
	with differential $\del^k_X: X^k \rightarrow X^{k-1}$ given by 
	\begin{equation}
		\begin{split}
			(\del^k_X \phi) (v,q_1,\ldots,q_{k-1},w) = &\phi(v,v,q_1,\ldots,q_{k-1},w) 
			+ \sum_{i=1}^{k-1} (-1)^i \phi(v,q_1,\ldots.q_i,q_i,\ldots,q_{k-1},w) \\
			&+(-1)^k \phi(v,q_1,\ldots,q_{k-1},w,w)
		\end{split}
	\end{equation}
\end{defn}

We have $X_k \cong \A^e \otimeshat \A^{\otimeshat k}$ for the completion in the projective topology of the  tensor product induced by the Fréchet topology of $\CCinf(V)$, because for the completed tensor product one has $\CCinf(V) \otimeshat \CCinf(V) = \CCinf(V\times V)$. For details  see \citep{jarchow} especially Section 21.6.
 
The $\A^e$-module structure is given by
\begin{equation}
	(a \chi)(v,q_1,\ldots.q_k,w) = a(v,w)\chi(v,q_1,\ldots,q_k,w) 
\end{equation}
for $a \in \A^e, \chi \in X_k$  and $v,w,q_1,\ldots,q_k \in V$, which corresponds to the algebra multiplication in the $\A^e$ factor of the $X_k$.

\begin{lemma}
	We get a  resolution of $\A$ as $\A^e$-module by an exact sequence
	\begin{equation} \label{bc}
		0 \longleftarrow \A \overset{\epsilon }{\longleftarrow} X_0  \overset{\del^1_X }{\longleftarrow} X_1 \leftarrow \cdots  \overset{\del^k_X }{\longleftarrow} X_k \longleftarrow \cdots, 
	\end{equation} 
	where 
	\begin{equation}
		(\epsilon a)(v) = a(v,v).
	\end{equation}
\end{lemma}
\begin{proof} 
	One easily sees that  $\del_X$ and $\epsilon$ are $\A^e$ linear and a computation shows that $\del_X \circ \del_X=0$ and $\epsilon  \circ \del_X=0$, so we really have  a complex.
			 
	To show that it is exact one uses the homotopies $h^k_X : X_k \rightarrow X_{k+1}$
	\begin{align}
		(h^{-1}_X a)(v,w)                     & = a(v)                               \\
		(h^k_X \chi )(v,q_1,\ldots,q_{k+1},w) & = -(-1)^k \chi(v,q_1,\ldots,q_{k+1}) 
	\end{align}
	For details see \citep[Ch.3]{art}.
\end{proof}

We note that the homotopies $h^k_X$ are not $\A^e$ linear, which will cause some trouble later.

\begin{remark}
	This resolution is topologically free, but not in the purely algebraic setting. This is one reason, why one cannot simple apply the standard techniques of homological algebra. For the continuous case one could still use them, see \citep{pflaum,connes}, however  not for the differential cohomology.
\end{remark}

\subsection{Koszul complex}

Next we need another complex, which cannot be defined for $\CCinf(M)$ for an arbitrary manifold $M$, but only for the special case of a convex subset of $\R^n$. However, we will later be able to compute the Hochschild cohomology for arbitrary manifolds by localizing to convex sets.  In the definition of the Kozsul complex and the related chain maps we follow \citep[Sect.5.4]{weisphd}.

Let $\A= \CCinf(\R^n)$ or $\A= \CCinf(V)$ where $V  \subset \R^n$ is a convex open set. For a (finite dimensional) vector space $W$ we denote by  $\Lambda^\bullet(W)$ the antisymmetric tensor algebra over $W$, and by $W^*$ its dual.
\begin{defn}
	We define the Koszul complex $(K,\del_K)$ over  $\A$  as
	\begin{align}
		K_0 & = \A^e                                                                     \\
		K_k & = \A^e \otimes \Lambda^k(\R^n)^* \cong \CCinf(V\times V,\Lambda^k(\R^n)^*) 
	\end{align}
	Also every  $K_k$ has an $\A^e$-module structure by multiplication in the first factor.
				
	Next we define the differential $\del^k_K : K_k \rightarrow K_{k-1}$ by
	\begin{equation}
		(\del_K \omega )(v,w)(x_1,\ldots, x_{k-1}) = (\omega (v,w))(v-w,x_1,\ldots, x_{k-1})
	\end{equation} 
	for $\omega \in K_k$ and $v,w \in V, x_i \in \R^n$.
\end{defn}

Note that for the definition of the differential we need to insert $v-w$, which is actually a point on the manifold $V$, into a form. This is one reason, why one can define the Koszul complex only for a subset of $\R^n$. Actually it would be enough to consider $V=\R^n$ since every convex open set is diffeomorphic to $\R^n$.

We get the following finite and free resolution 
\begin{equation} \label{kc}
	0 \longleftarrow \A \overset{\epsilon }{\longleftarrow} K_0  \overset{\del^1_K }{\longleftarrow} K_1 \leftarrow  \cdots  \overset{\del^k_K }{\longleftarrow} K_n 
\end{equation}

Again one has to check that $\del_i \circ  \del_{i+1} =0$, which is a straightforward calculation using the fact that one has to insert the same argument $v-w$ twice. 

\begin{lemma}
	The sequence  \eqref{kc} is exact.
\end{lemma}
\begin{proof}
	Use the homotopies $h^k_K : K_k \rightarrow K_{k+1}$
	\begin{equation}
		(h^k_K \omega)(v,w) = - \sum_{j=1}^n e^j \wedge \int t^k \frac{\del \omega}{\del w^j} (v, tw +(1-t)v) \d t 
	\end{equation}
	For details see \citep[Sect. 5.4]{weisphd}.
\end{proof}

Note that for this we needed the convexity of $V$ and the completion of the tensor product.

Using  $\xi_i = x_i \otimes 1 - 1 \otimes x_i \in \A^e$, we can write the differential on forms $e^I = \left({e^1} \right)^{\wedge I_1} \wedge \cdots  \wedge \left({e^n}\right)^{\wedge I_n}$ with a multiindex $I \in \Z^n$, where all $I_j $ can be assumed to be  0 or 1, because otherwise it would vanish,  as 
\begin{equation}
	\del_K e^I = \sum_i \xi_i \otimes  \iota_{e_i} e^I.
\end{equation}
This can be easily seen using the fact that $\delta_K e^i = \xi_i$.

Next we define $F^k: K_k \rightarrow X_k$
\begin{equation}
	F^k(\omega)(v,q_1,\ldots,q_k,w)=  \omega(v,w)(q_1-v,\ldots, q_k - v) 
\end{equation}

and $G^k: X_k \rightarrow K_k$
\begin{align*}
	(G^k \chi)(v,w) & = \sum_{i_1\ldots i_k=1}^n e^{i_1}\wedge \cdots \wedge e^{i_k} \int\limits_0^1 \d t_1 
	\cdots \int\limits_0^{t_{k-1}} \d t_k   \\
	                & \frac{\del^k \chi}{\del q_1^{i_1} \cdots \del q_k^{i_k}}                              
	(v,t_1 v +(1- t_1)w, \ldots, t_k v +(1- t_k)w, w).
\end{align*}
The definition of $F$ appears first in \citep[Sect.III.2$\alpha$]{connes} and the $G$ originates in \citep{bordemann}.
The maps $F$ and $G$ are chain maps and $\A^e$-module homomorphisms, this means we get the following commutative diagram of $\A^e$-linear maps:

\begin{tikzpicture}
	[node distance=1.5cm]
	\node(x0) {$0$};
	\node(x1) [right=of x0] {$\A$};
	\node(x2) [right=of x1]{$\X_0$};
	\node(x3) [right=of x2]{$\cdots$};
	\node(x4) [right=of x3]{$X_k$};
	\node(x5) [right=of x4]{$X_{k+1}$};
	\node(x6) [right=of x5]{$\cdots$};
	\node(k0) [below=of x0]  {$0$};
	\node(k1) [right=of k0]{$\A$};
	\node(k2) [right=of k1]{$K_0$};
	\node(k3) [right=of k2]{$\cdots$};
	\node(k4) [right=of k3]{$K_k$};
	\node(k5) [right=of k4]{$K_{k+1}$};
	\node(k6) [right=of k5]{$\cdots$};
	\path[->] (x1) edge (x0)
	(x2) edge node[auto] {$\epsilon$} (x1)
	(x3) edge node[auto] {$\del^1_X$} (x2);
	\path[->] (x4) edge node[auto] {$\del^{k}_X$} (x3)
	(x5) edge node[auto] {$\del^{k+1}_X$} (x4)
	(x6) edge node[auto] {$\del^{k+2}_X$} (x5);
	\path[->] (k1) edge  (k0)
	(k2) edge node[auto] {$\epsilon$} (k1)
	(k3) edge node[auto] {$\del^1_K$} (k2);
	\path[->] (k4) edge node[auto] {$\del^{k}_X$} (k3)
	(k5) edge node[auto] {$\del^{k+1}_K$} (k4)
	(k6) edge node[auto] {$\del^{k+2}_K$} (k5);
	\draw[double] (x1) -- (k1);
	\draw[double] (x2) -- (k2);
	\path[->] (x4.-80) edge [left] node{$F^k$} (k4.90);
	\path[->] (k4.110) edge [right] node{$G^k$} (x4.-100);
	\path[->] (x5.-80) edge [left] node{$F^{k+1}$} (k5.90);
	\path[->] (k5.110) edge [right] node{$G^{k+1}$} (x5.-100);
\end{tikzpicture}

One can show that
$G^k \circ F^k = \id$
which proves that
$\Theta^k = F^k \circ G^k $
is a projection.
Further one can compute explicitly 
\begin{equation}
	\begin{split}
		(\Theta^k \chi)(v,q_1,\ldots,q_k,w) = \sum_{i_1\ldots i_k=1}^n \sum_{\sigma\in S_k} \sign(\sigma) (q_1-v)^{i_{\sigma(1)}} \cdots (q_k-v)^{i_{\sigma(k)}} \\ 
		\int_0^1 \d t_1 \cdots \int_0^{t_{k-1}} \d t_k
		\frac{\del^k \chi}{\del q_1^{i_1} \cdots \del q_k^{i_k}} (v,t_1 v +(1- t_1)w, \ldots, t_k v +(1- t_k)w, w),
	\end{split}
\end{equation}
where $S_k$ is the symmetric group with $k$ elements, and the upper indices on the brackets denote the components.

\begin{remark}
	The explicit homotopies $F,G$ and $\Theta$ would not be necessary in the completely algebraic context, because their existence can be proven in a completely abstract way, but we need them here to make sure that everything  stays in the continuous or differential  Hochschild cohomology.
\end{remark}

\begin{lemma} \label{th:sk}
	There exists a homotopy between $\Theta^\bullet$ and $\id_{X_\bullet}$.
\end{lemma}
\begin{proof}
	See \citep[Proposition 5.7.2]{weisphd}
\end{proof}

We now consider the vector space $\Hom^\mathrm{cont}_{\A^e}(X_k,\M)$ of continuous $\A^e$-linear maps.
With the pullback of the differentials $\delta^{k}_X = (\del^k_X)^*$, defined by $(\del^* \phi)(a) = \phi(\del a)$ for $\phi\in\Hom^\mathrm{cont}_{\A^e}(X_k,\M)$ and $a \in X_\bullet)$, we get the complex $(\Hom^\mathrm{cont}_{\A^e}(X_\bullet,\M), \delta_X)$.

\begin{prop}
	The complexes $(\Hom^\mathrm{cont}_{\A^e}(X_\bullet,\M), \delta_X)$ and $\HC^k_\mathrm{cont}(\A,\M)$ are isomorphic with isomorphism $\Xi: \Hom^\mathrm{cont}_{\A^e}(X_\bullet,\M) \rightarrow \HC^k_\mathrm{cont}(\A,\M)$
	\begin{equation}
		(\Xi^k \psi)(a_1,\ldots,a_k) = \psi(1 \otimes a_1 \otimes \cdots a_k \otimes 1).
	\end{equation}
\end{prop}
\begin{proof}
	$\Xi$ is  a chain map, since one can easily see that  
	\begin{equation}
		\Xi \circ \delta_X = \delta \circ \Xi.
	\end{equation}
	The map $\Theta$ is an isomorphism, because of the universal property of the tensor product for continuous maps. The inverse is given by
	\begin{equation}
		\Xi^{-1}(1 \otimes a_1 \otimes \cdots \otimes a_k \otimes 1) = \psi(a_1, \ldots, a_k).
	\end{equation}
	For details see \citep[Prop.5.2.1]{weisphd}.
\end{proof}

We have a well-defined differential subcomplex $\HH^\bullet_\diff (\A, \M) $ in the case of a differential bimodule. But we also want to define a complex $\Hom^\diff_{\A^e}(X_k,\M)$, with which we can compute this differential Hochschild cohomology. 
For this we  set 
\begin{equation}
	\Hom^{\diff,L}_{\A^e} (X_k,\M) = (\Xi^k)^{-1} (\DO^L( \A,\M))
\end{equation}
Since $\Xi$ is a chain map we also get a well defined subcomplex $\Hom^\diff_{\A^e}(X_\bullet ,\M)$ of $\Hom^\cont_{\A^e}(X_\bullet ,\allowbreak \M)$. By  construction we get an isomorphism of complexes
\begin{equation}
	\Xi : (\Hom^\diff_{\A^e}(X_\bullet ,\M), \delta_X) \rightarrow (HC^\bullet_\diff(\A,\M),\delta)
\end{equation}

Since  $\M$ is a topological bimodule, we have that the map $(a,f,b) \mapsto a \bullet f \bullet b$ is continuous.
So by continuity we get an $\A^e$-module structure, for the completed tensor product, given by 
\begin{equation}
	(a \otimes b) \cdot f = a \cdot f \cdot b.
\end{equation}
This can also be written as 
\begin{equation}\label{diffmod}
	\hat a \bullet f = \sum_{|I|<l} (\Delta^*_0 \del_I \hat a) \cdot f^I 
\end{equation}
for all $\hat a \in \A^e$. Here $\Delta^*_k$ denotes the pull-back with the total diagonal map $\Delta_k : V \rightarrow V^{k+2}$ and the differentiation acts on the second argument of $\hat a$.

Using the local form of a differential operator it is also possible to get an explicitly form of the elements of $\Hom^{\diff,L}_{\A^e}(X_k,\M)$.

\begin{lemma}[{\citep[Lemma 5.3.6]{weisphd}}]\label{th:lochom} 
	An element $\psi \in \Hom^{\diff,L}_{\A^e}(X_k,\M)$ has the form 
	\begin{equation}
		\psi(\chi)= \sum_{|I_1|<l_1, \ldots,\abs{I_k}<l_k,\abs{J}<l} \left(\Delta^*_k \frac{\del^{\abs{I_1}+ \cdots \abs{I_k}} \chi}{\del q^{I_1}_1 \cdots \del q^{I_k}_k \del w^J} \right) \cdot \psi^{I_1 \cdots I_k J} 
	\end{equation}
	with multiindices $I_1,\ldots I_k, J \in \N_0^n$ and $\psi^{I_1 \cdots I_k J} \in \M$, and $l$ the order of the differential bimodule.
\end{lemma}

With this one can show that the constructed homotopies all respect the differential subcomplex in the following sense:

\begin{prop}[{\citep[Prop. 5.7.3]{weisphd}}]\label{th:kdiff}
	The pullbacks $(G^k)^*: \Hom_{\A^e}(K_k,\M) \rightarrow \Hom_{\A^e}^{\mathrm{diff},L}(X_k,\M)$ only take values in the differential cochains of multiorder $L=(l+1.\ldots,l+1)$  and
	\begin{equation}
		(\Theta^k)^* : \Hom_{\A^e}^{\mathrm{diff}}(X_k,\M) \rightarrow \Hom_{\A^e}^{\mathrm{diff},L}(X_k,\M),
	\end{equation}
	so elements of the differential Hochschild complex are mapped into such elements.
	Also for all $L \in \N^{k+1}_0$ we have
	\begin{equation}
		(s^k)^* : \Hom_{\A^e}^{\mathrm{diff},L}(X_{k+1},\M) \rightarrow \Hom_{\A^e}^{\mathrm{diff},\tilde L}(X_k,\M),
	\end{equation}
	where $\tilde l_i = (k-1)! + \abs{L} +l$.
\end{prop}

We want to compute explicitly the  map $\tilde G :\Hom(K,\M) \to \HC(A,M)$, which is induced by $G^*$, in the case of a symmetric bimodule.
We get 
\begin{equation} \label{eq:gtilde}
	\tilde G( \phi)(a_1,\dots, a_k) = \sum_{i_1,\dots i_k} (\del_{i_1} a_1)\dots(\del_{i_k} a_k) \phi(e^{i_1}\wedge \dots \wedge e^{i_k}).
\end{equation}

\begin{prop}\label{th:kohh}
	We have the following isomorphisms of complexes:
	\begin{equation}
		\HH^\bullet_\diff(\A,\M) \cong \H(\Hom^\diff_{A^e}(X_\bullet,\M)) \cong \H(\Hom_{\A^e}(K_\bullet,\M))
	\end{equation}
\end{prop}

\begin{remark} \label{th:homkoz}
	Since $K^k$ is free and finite dimensional as an $\A^e$-module for any $k\in \N$, we have that $\Hom_{\A^e}(K_k,\M) \cong K_k^* \otimes_{\A^e} \M  \cong (\A^e \otimes \Lambda^\bullet(\R^n)) \otimes_{\A^e} \M \cong \Lambda^\bullet(\R^n) \otimes \M$, for any module $\M$, where $K_k^*$ is the dual of $K_k$ as $\A^e$-module.
\end{remark}

The differential on $\Hom(K_k,M)$ can then  be written as 
\begin{equation}
	\delta_K (\phi \otimes f) = \xi_i e^i \wedge \phi \otimes f
\end{equation}
for $\phi \otimes f \in \Lambda^\bullet(\R^n) \otimes \M$.

Since the Koszul complex is finite and every $K_k$ is also a finite dimensional module 
it is much smaller than the bar complex. So it is easier to handle, but still big enough to compute the desired Hochschild cohomology.
For defining the Koszul complex one needs to use the completion of the tensor product in $\A^e$ because otherwise it is not possible to define for example the homotopy.

\subsection{Generalisation of the HKR theorem}

The aim of this section is to prove a generalization of the HKR theorem.
We start with the simple case that the considered manifolds are $\R^n$.

\begin{theorem}\label{th:hkrl}
	Consider an arbitrary smooth map $\R^n \xrightarrow{p} \R^m$ between $\R^n$ and $\R^m$.
	Then 
	\begin{equation}
		\HH^\bullet_\diff(\CCinf(\R^m),\CCinf(\R^n)) = \Lambda^\bullet(\R^{m}) \otimes \CCinf(\R^n) 
	\end{equation}
	as $\CCinf(\R^m)$-bimodules.
\end{theorem}
\begin{proof}
	Using \cref{th:kohh} we compute the cohomology of the corresponding Koszul complex.
	We consider an element $ e^I \otimes f \in \Hom(K_k,\CCinf(\R^n))$, where $I$ is a  multiindex. Since $K_k$ is  a free $\A^e$-module this is a generating set.
	We have 
	\begin{equation}
		\del_K ( e^I \otimes f )=  \xi_i e^i \wedge  e^I \otimes f = e^i \wedge e^I \otimes (\pr^* x^i f - f  \pr^* x^i) =0,
	\end{equation}
	using the fact that the tensor product is $\A^e$-linear and the fact that the multiplication in 
	$\CCinf(\R^n)$ is commutative.
	So the differential is trivial and \cref{th:homkoz} gives us the desired result.
\end{proof}

\begin{remark}
	Note that we only needed the fact that $\CCinf(\R^n)$ is a symmetric bimodule, so for any symmetric module $\M$ we get
	\begin{equation}
		\HH^\bullet(\CCinf(\R^m),\M) \cong \Lambda^\bullet(\R^{m}) \otimes \M.
	\end{equation} 
\end{remark}

Next we want to consider the trivial situation for the Hochschild cohomology of the  differential operators.

\begin{theorem}\label{th:hkrld}
	Let $\R^n \xrightarrow{p} \R^m$ be the projection on the first $k$ coordinates.
	Then 
	\begin{equation}
		\HH^\bullet_\diff(\CCinf(\R^m),\DO(\R^n)) \cong \Lambda^\bullet(\R^{m-k}) \otimes \DOver(\R^n)
	\end{equation}
	as $\CCinf(\R^m)$-modules.
\end{theorem}
\begin{proof}
	This proof is similar to the proof of \cref{th:hkrl}. 
	Considering elements of the form $e^I \otimes f y^J   \in \Hom(K_k,\DO(\R^n))$, where $y^I$ is a symbol, which we identify with the corresponding differential operator, $f \in \CCinf(\R^n)$ and $I$, $J$ are multiindices. We also assume that $y^I$ acts on everything to the right. Again elements of this form generate the whole  of $\Hom(K_k,\DO(\R^n))$. 
	We have
	\begin{align}
		\del_K (e^I \otimes f y^J ) & = \sum_{i=1}^n  \xi_i e^i \wedge e^I \otimes  f y^j                        \\
		                            & =\sum_{i=1}^n  e^i \wedge e^I \otimes  f ( \pr^* x^i y^J -  y^J \pr^* x^i) \\
		                            & = \sum_{i=1}^k  e^i \wedge e^I \otimes f [ x^i, y^J ]                      \\ 
		                            & = - \sum_{i=1}^k  e^i \wedge e^I \otimes f \del_{y^i} y^J                  
	\end{align}
	using $\pr^* x^i =0$ for $i >k$ and $[x^i,y^J] = \del_{y^i} y^I$.
	In this case $\DOver$ are those differential operators, whose symbols only contain $y^i$ with $i>k$.
				
	For $\del_K  ( e^I \otimes f y^J ) =0$ we need that $\del_K  e^{I'} \otimes  y^{J'}  =0$ where $I'\in \N^k$ consists of  the first $k$ entries of $I$ and similarly for $J$.  This can be considered as the de-Rahm differential on $\R^k$ for polynomial functions. The cohomology of this is known  to  be trivial except  in degree 0, where it is $\C$ and the non trivial element is 1. Since the differential is trivial on the other part, we get the result.
\end{proof}

Now we want to use this result for $\R^n$ and generalize it for the situation of an arbitrary smooth map $\pr: M \rightarrow N$ between two manifolds. To be able to localize things we need the assumption that $\pr(N)$ is a submanifold of $M$. 

\begin{remark}
	Since $\pr(N)$ is a submanifold of $M$, we can assuming that $\pr$ has constant rank, since this is true for every connected component. With the constant rank theorem we get adapted charts. This means for every point $p \in P$ there are  open sets $ p \in V \subset P$  and $ \pr(p) \in U \subset M$, with $\pr(V) =U$, and diffeomorpism $V  \rightarrow \tilde V \subset \R^{n}$ and $U \rightarrow \tilde U \subset \R^m$ such that in this charts $\pr$ is the projection on the first $k = \rang(\pr)$ components. Furthermore we can assume that $\tilde U$ and $\tilde V$ are convex.
\end{remark}

\begin{lemma}\label{th:emcm}
	The restrictions and charts shown in the following diagram are chain maps
				
	\begin{tikzpicture}
		\node(a) {$\HC^\bullet_\diff(M,\DO(P))$};
		\node(b) [below=of a] {$\HC^\bullet_\diff (U,\DO(\pr^{-1}(U))$};
		\node(c) [below=of b] {$\HC^\bullet_\diff(U,\DO(V))$};
		\node(d) [right=of c] {$\HC^\bullet_\diff(\tilde U,\DO(\tilde V))$};
		\path[->] (a) edge   (b)
		(b) edge (c) 
		(c) edge  node[above] {$\cong$} (d);
	\end{tikzpicture}
\end{lemma}
\begin{proof}
	This follows from the fact that all involved operators are local. 
\end{proof}

Now we get the the fist of the two main results of this part of this paper, which gives a generalization of the HKR theorem.  A similar statement using the same concepts for the proof is given in \citep{bordemann}.

To simplify the notation we will sometimes write $\HC^\bullet(M,\cdot)$ instead of $\HC^\bullet(\CCinf(M), \cdot)$.

\begin{theorem}\label{th:hkr}
	Let $N \xrightarrow{\pr} M$ be such that $ \pr(N)$ is a closed submanifold of $M$ then
	\begin{equation}
		\HH^\bullet(\CCinf(M),\CCinf(N)) = \X^\bullet(M)|_{\pr(N)} \otimes_{\CCinf(M)} \CCinf(N)
	\end{equation}
	as $\CCinf(M)$-bimodule.
\end{theorem}
\begin{proof}
	First we check that if we take $M,N$ and $\pr$ as in  \cref{th:hkrl} we get the same statement as there.
	Since in this case we have global charts, we have $\X^\bullet(M) \cong \CCinf(M) \otimes \Lambda^\bullet(\R^m).$ So we get $\X^\bullet(M)|_{\pr(N)} \otimes_{\CCinf(M)} \CCinf(N) \cong (\CCinf(\pr(N)) \otimes \Lambda^\bullet(\R^m)) \otimes_{\CCinf(M)} \CCinf(N) \cong \Lambda^\bullet(\R^m) \otimes \CCinf(N)$. The last isomorphism holds since $\CCinf(\pr(N))$ is a subalgebra of $\CCinf(M)$, since $\pr(N)$ is closed, so any function on $\pr(N)$ can be extended to a function on $M$.
				
	The idea is to localize things such that \cref{th:hkrl} can be applied, and then glue them together again.
	Since we consider the differential Hochschild cohomology it is enough to consider an open neighborhood of $\pr(N)$ in $M$. So given an atlas  $\{U_\alpha\}$ of submanifold charts of $\pr(N)$ we can assume  w.l.o.g. that $\bigcup_\alpha U_\alpha =M$, since $M \setminus \pr(N)$ is open we can take this is a submanifold chart and this to get a global atlas of $M$. We also consider a locally finite partition of unity $\chi_\alpha$ subordinate to $\{U_\alpha \}$  and an atlas $\{V_\alpha \}$ of  $N$, with partition of unity $\psi_\alpha$.
	These are adapted in the sense that $\pr (V_\alpha) = U_\alpha |_{\pr(N)}$
				
	Now consider a  $\phi \in \HC^l_\diff(M,\CCinf(N))$ which is closed.
	With \cref{th:emcm}  the restrictions $\phi_{V_\alpha} \in \HC^l_\diff(U_\alpha,\CCinf(V_\alpha))$ are closed. With the first part of the proof there exists $\sigma_\alpha \in \Lambda^\bullet(\R^m) \otimes \CCinf(N)$ and $\theta_\alpha \in \HC^{l-1}(U_\alpha,\CCinf(V_\alpha))$ with $\phi_{V_\alpha} = \sigma_\alpha + \delta \theta_\alpha$.  
	The restrictions
	\begin{equation}
		\tilde \theta_\alpha (a_1, \ldots , a_k)|_{V_\alpha}  = \psi_\alpha \theta_\alpha(a_1|_{U_\alpha}, \ldots a_k|_{U_\alpha})
	\end{equation}
	and $0$ elsewhere, define global elements $\theta_\alpha$, and similarly one can define global elements $\sigma_\alpha$. Clearly we have $\delta \tilde \theta_\alpha + \tilde \sigma_\alpha = \psi_\alpha  (\delta \theta + \sigma)$, 
	and, since $\psi_\alpha$ is locally finite, we get that $\theta = \sum_\alpha \tilde \theta_\alpha$ and 
	$\sigma= \sum_\alpha \tilde \sigma_\alpha$ are well-defined differential operators, and we also get
	\begin{equation}
		\phi = \sum_\alpha \psi_\alpha \phi =  \sum_\alpha ( \tilde \sigma_\alpha  + \delta \tilde \theta_\alpha)
		= \sigma +  \delta \theta 
	\end{equation}
	This gives the desired result.
\end{proof}

\begin{remark}
	The isomorphism in the previous theorem is given by the pullback of $\Theta$ since the differential in the Kozsul complex is trivial. From \cref{th:kdiff} it also follows that the image of $\Theta^*$ are exactly the multivector fields, because the module is symmetric, so it is a differential bimodule of order $l=0$. So the pullback maps into the totally antisymmetric multidifferential operators of order one in each argument.
\end{remark}

\begin{prop}\label{th:antisym}
	The map which assigns to every cocycle its cohomology class is given by the total antisymmetrization.
\end{prop}
\begin{proof}
	The map $\tilde G$ in \cref{eq:gtilde} is an isomorphism on cohomology. So let $\phi \in \HC(M,\M)$ be a cocycle. Then there exists an $\eta \in \Hom(K,\M)$ such that $[\phi] = [\tilde G \eta]$. So we have  $\phi = \tilde G \eta + \del \psi$ for a $\psi \in \HC(M,\M)$.  With this we get $\Alt(\phi) = \Alt(\tilde G \eta )  + \Alt(\del \psi) = \tilde G \eta$, since $\tilde G\eta$ is antisymmetric and $\Alt \delta =0$, since the algebra is commutative. 
\end{proof}

From this we easily get the classical HKR theorem.
\begin{corollary}
	For a manifold $M$ we have
	\begin{equation}
		\HH^\bullet_\diff(\CCinf(M)) \cong \X^\bullet(M)
	\end{equation}
\end{corollary}
\begin{proof}
	Use \cref{th:hkr} with $N = M$ and $\pr = \id$.
\end{proof}

We want to explicitly compute $\HH^\bullet_\diff(\CCinf(M),\CCinf(M))$ in the low degrees:

For $f \in \HC^0_\diff(\CCinf(M),\CCinf(N)) \cong \CCinf(N)$ we have 
\begin{equation}
	(\delta f)(a) = \pr^*a f -f \pr^*a  =0
\end{equation}
for all $a \in \CCinf(M)$ and $f \in \CCinf(P)$. So every element of $\HC^0_\diff(\CCinf(M),\CCinf(N))$ is closed but since 
there are no elements of degree $-1$, we have $\HH^0_\diff(\CCinf(M),\CCinf(N)) = \CCinf(N)$.

For $\phi \in \HC^1_\diff(\CCinf(M),\CCinf(N))$, we have 
\begin{equation}
	(\delta \phi)(a,b) = \pr^*a \phi(b) - \phi(ab) + \phi(a) \pr^* b.
\end{equation}
So $\phi$ is closed if 
\begin{equation}\label{hc1}
	\phi(ab) = \pr^*a \phi(b) + \phi(a) \pr^* b,
\end{equation}
which means that $\phi$ is a derivation. Since $\delta^0 =0$ there are no exact elements, and the cohomology  is given by the elements satisfying \eqref{hc1}.

Before proving the main theorem we need a small lemma:

\begin{lemma}\label{th:factor}
	Let $V$ be  a finite dimensional vector space  and $W \subset V$ be a vector subspace then 
	\begin{equation}
		\factor{\Lambda^\bullet(V)}{\sprod{\Lambda^1(W)}} = \Lambda\left(\factor{V}{W}\right).
	\end{equation}
	Here $\sprod{x}$ denotes the ideal generated by $x$.
\end{lemma}
\begin{proof}
	We define a homomorphism $\phi: \factor{\Lambda^\bullet(V)}{\sprod{\Lambda^1(W)}} \rightarrow \Lambda^\bullet(\factor{V}{W})$ by $[v_1 \wedge \cdots \wedge v_k] \mapsto [v_1] \wedge \cdots \wedge [v_k]$, where $[\cdot]$ denotes the corresponding equivalence classes.
	First of all it is easy to see that this is well defined, since any $X \in \factor{\Lambda^\bullet(V)}{\sprod{\Lambda^1(W)}}$ contains a $w \in W$ and $w =0$ in    $\factor{V}{W}$. $\phi$ is clearly surjective.  Using a basis $\{e_i\}_{i \in I}$ such that $\{e_i\}_{i\in J}$, with $J \subset I$, is a basis of $W$, one gets that $\phi(e_{i_1} \wedge \cdots \wedge e_{i_k}) =0$ if and only if one of the $i_j$ is in $J$. This shows $\phi$ to be injective. So $\phi$ is a isomorphism.
\end{proof}

Now we can prove the main theorem of this paper, namely the computation of the Hochschild cohomology $\HH^\bullet(\CCinf(M),\DO(N))$. It is a significant generalization of the theorem given in \citep{art}, where the situation of a fibered manifold is considered. The big difference to our situation is that there the cohomology is trivial except in degree zero, while here it is in general always non trivial. 
We need the assumption that $\pr(N)$ is a closed submanifold. This is needed to use the local situation given in \cref{th:hkrld}.  The fact that $\pr(N)$ is closed is important, because otherwise $\CCinf(\pr(N))$ would not be a subalgebra of $\CCinf(M)$, which is important for our construction. 

After proving this theorem, we want to give some details on the isomorphism given in it.

\begin{theorem} \label{th:hkrd}
	Let $N \xrightarrow{\pr} M$ be such that $ \pr(N)$ is a closed submanifold of $M$ then
	\begin{equation}
		\HH^\bullet_\diff(\CCinf(M),\DO(N)) \cong \factor{\X^\bullet(M)|_{\pr(N)}}{\sprod{\X(\pr(N))}} \otimes_{\CCinf(M)} \DOver(N)
	\end{equation}
	as $\CCinf(M)$-bimodule, where $\sprod{x}$ denotes the ideal generated by $x$.
\end{theorem}
\begin{proof} 
	Again we first compare the statement of this theorem with the local situation in \cref{th:hkrld}, i.e. $M= \R^m$ and $N = \R^n$. 
	In this case we have $\pr(N) = \R^k$ so $\X^\bullet(M)|_{\pr(N)} = \CCinf(\R^k) \otimes \Lambda^\bullet (\R^m)$ and $\X(\pr(N)) = \Lambda^\bullet(\R^n) \otimes \CCinf(\R^n)$. These equalities follow easily form the fact that the multivector bundle over $\R^n$ is trivial. 
	Next with \cref{th:factor} we have that $\factor{\Lambda^\bullet(\R^m)}{\sprod{\Lambda^\bullet(\R^k)}} = \Lambda^\bullet(\R^{m-k})$.
	So we get 
	\begin{align*}
		\factor{\X^\bullet(M)|_{\pr(N)}}{\sprod{\X(\pr(N))}} & =                                                                       
		\factor{\CCinf(\R^k) \otimes  \Lambda^\bullet(\R^m)}{\sprod {\CCinf(\R^k) \otimes  \Lambda^1(\R^k)}} \\
		                                                     & = \CCinf(\R^k) \otimes  \factor{\Lambda^\bullet(\R^m)}{\Lambda^1(\R^k)} \\
		                                                     & = \CCinf(\R^k) \otimes \Lambda^\bullet(\R^{m-k}),                       
	\end{align*}
	since we consider $	\X^\bullet(M)|_{\pr(N)}$ and $\sprod{\X(\pr(N))}$ as $\CCinf(M)$-modules. 
				
	Globalizing works as in the previous theorem.
\end{proof}

\begin{remark}
	The submodule  $\HC^\bullet(\CCinf(M),\DOver(N))$ is symmetric by the definition of a vertical operator and so similarly to above the pullback of $\Theta$ on this submodule maps into the multivector fields.
	If one chooses a connection on $N$ as a fibered manifold over $\pr(N)$ one gets 
	\begin{equation}
		\DO(N) = \DOver(N) \oplus \DOhor(N)
	\end{equation}
	So one also has $$\HC^\bullet(\CCinf(M),\DO(N)) =  \HC^\bullet(\CCinf(M),\DOver(N)) \oplus \HC^\bullet(\CCinf(M),\DOhor(N)).$$
	Now using the Koszul complex one  can see that any closed element of $\HC^\bullet(\CCinf(M),\allowbreak \DOhor(N))$ is exact, since with the notation as in the proof of \cref{th:hkrld} we have that $y^{I'}$ is non constant. 
	So in every cohomology class their is a representative which lies in $\HC^\bullet(\CCinf(M),\DOver(N))$  and on this the isomorphism in \cref{th:hkrd} is given by the antisymmetrization, which can be shown as \cref{th:antisym}.
\end{remark}

\begin{remark}[Connection to bimodule deformation]
	This cohomology group in degree two gives the obstruction for the existence of a $\CCinf(M)$-module deformation of $\CCinf(N)$, see \cref{th:mho}. What we see is that this deformation only always exists if $\dim M = \dim \pr(N)$ or $\dim M = \dim \pr(N) +1 $.  In all other cases one has to expect obstructions. 
				   
	The existence of a bimodule deformation of a fibered manifold $P \xrightarrow{p} M$ as in \cref{ch:bim} would be granted if  the cohomology of $\DO(P)$ as $\CCinf(M) \otimes \CCinf(M) \cong \CCinf(M \times M)$-bimodule would be trivial, where $\pr $ is the projection $p:P \rightarrow M$ composed with the diagonal.  Note however that the previous theorem cannot be used directly, because for the bimodule deformation we need the algebraic tensor product and for the isomorphisms of this section the topological one.
\end{remark}

We want to further interpret the cohomologies, which we computed to be $\X^\bullet(M)|_{\pr(N)} \allowbreak \otimes_{\CCinf(M)} \CCinf(N)$ resp.$\factor{\X^\bullet(M)|_{\pr(N)}}{\sprod{\X(\pr(N))}} \otimes_{\CCinf(M)} \DOver(N)$, because they look not very intuitive at first glance.
So want to show that in fact this two can be interpreted as vector bundles over $N$.
  
First we consider the simpler case of $\X^\bullet(M)|_{\pr(N)} \otimes_{\CCinf(M)} \CCinf(N)$.
We have
\begin{equation}
	\X^\bullet(M) \otimes_{\CCinf(M)} \CCinf(N) \cong 
	\X^\bullet(M)|_{\pr(N)} \otimes_{\CCinf(\pr(N))} \CCinf(N),
\end{equation}
since $\CCinf(M) = \CCinf(\pr(N)) \oplus N$, where $N= \{a \in \CCinf(M) \;|\; a|_N = 0 \}$, because we assume $\pr(N)$ to be closed. But the direct sum is not canonical because one has to embed $\CCinf(\pr(N))$ in $\CCinf(M)$. One possibility is defining a prolongation $\prol: \CCinf(\pr(N)) \rightarrow \CCinf(M)$, which satisfies $(\prol a)|_N  = a$. One has $\pr^*a = 0$   for $ a\in N$. This can be done for example by choosing a tubular neighborhood.
  
In the following proposition we need the concept of the pullback of a vector bundle, see e.g. \citep[Section {III,8.9}]{michor}. For a vector bundle $E$ we denote the pullback along $f$ by $f^\sharp E$.
  
\begin{prop}\label{th:vecpb} 
	In the considered situation we have
	\begin{equation}
		\X^\bullet(M) \otimes_{\CCinf(M)} \CCinf(N) \cong \Gamma^\infty(\pr^\sharp \Lambda^\bullet(TM|_{\pr(N)})). 
	\end{equation} 
\end{prop}
\begin{proof}
	Since $\mathcal{E} = \X^\bullet(M)|_{\pr(N)}$ are the section of a $\pr(N)$ vector bundle it is a projective module over $\A= \CCinf(\pr(N))$. So their exists a projector $P \in \A^{n \times n}$  such that $\mathcal{E} = P \A^n$. So we have 
	\begin{equation}
		P \A^n \otimes_\A \CCinf(N)= \pr^*(P) \CCinf(N)^k.
	\end{equation}
	This shows $\X^\bullet(M) \otimes_{\CCinf(M)} \CCinf(N)$ to be projective as a $\CCinf(\pr(N))$-module. So it is a isomorphic to the sections of a vector bundle over $N$. 
	One can define a map $\phi: \X^\bullet(M) \otimes_{\CCinf(M)} \CCinf(N) \rightarrow \Gamma^\infty(\pr^\sharp \Lambda^\bullet(TM|_{\pr(N)})) $ by 
	\begin{equation}
		X \otimes f \mapsto f \pr^\sharp X.
	\end{equation}
	This clearly linear with respect to $\CCinf(N)$, so it is a vector bundle morphism. One can also show that $\phi$ is isomorphism. 
\end{proof}

Now we come to the case of $\factor{\X^\bullet(M)|_{\pr(N)}}{\sprod{\X(\pr(N))}} \otimes_{\CCinf(M)} \DOver(N)$. 
  
We recall that $\X^\bullet(M)|_{\pr(N)}$ and $\X^\bullet(\pr(N))$ can be considered as  vector bundles over $\pr(N)$, which is by assumption a manifold  and $\X^\bullet(M)|_{\pr(N)}$ is a subbundle of $\X^\bullet(\pr(N))$, so the quotient is again a vector bundle over $\pr(N)$.
   
For a manifold $M$ and a submanifold $ N \subset M$ we  define 
\begin{equation}
	\No(M,N) =  \factor{\X^\bullet(M)|_{N}}{\sprod{\X(N)}}
\end{equation}
to be the section of the exterior algebra of the normal bundle of $N$ in $M$. 
This means 
\begin{equation}
	\No(M,N) \cong \Gamma^\infty( \Lambda^\bullet(TN^\bot))
\end{equation}
as vector bundle over $N$. Here $TN^\bot = \factor{TM|_N}{TN}$ denotes the normal bundle of $N$. This can been seen using \cref{th:factor}.  
So we have 
\begin{equation}
	\HH^\bullet(\CCinf(M),\DO(P)) \cong \No(M,\pr(N)) \otimes_{\CCinf(M)} \DOver(N). 
\end{equation}  
First we note that $\CCinf(M) = \CCinf(\pr(N)) \oplus \mathcal{B}$, and for $b \in \mathcal{B}$ we have $\pr^*b = 0$. So we have 
\begin{equation}
	\No(M,\pr(N))\otimes_{\CCinf(M)} \DOver(N) \cong 
	\No(M,\pr(N)) \otimes_{\CCinf(\pr(N))} \DOver(N)
\end{equation}
since $\mathcal{P} = \No(M,N)$ is a $\CCinf(\pr(N))$-module.  
  
Since $\mathcal{P}$ are the sections of a vector bundle over $\pr(N)$, it is a projective $\CCinf(\pr(N))$-module. This means we can write $\mathcal{P} = P\CCinf(\pr(N))^k$ for a projector $P \in \CCinf(\pr(N))^{k \times k}$. 
We then have $P\CCinf(\pr(N))^k \otimes_{\CCinf(\pr(N))} \CCinf(N) \cong \pr^* P \CCinf(N)^k$ for purely algebraic reasons. This shows $\No(N,M) \otimes_{\CCinf(M)} \DOver(N)$ to be a projective $\CCinf(N)$-module, so it is isomorphic to the section of  a vector bundle over $N$. 
  
\begin{prop}
	We have 
	\begin{equation}
		\begin{split}
			\No(M,\pr(N)) \otimes_{\CCinf(M)} \DOver(N) 
			&\cong  \Gamma^\infty \left(\pr^\sharp \Lambda^\bullet \left(\factor{TM}{T \pr(N)} \right) \otimes \Lambda^\bullet(\Sym VN) \right) \\
			& \cong 	\Gamma^\infty \left(\pr^\sharp \Lambda^\bullet T \pr N^\bot \right) \otimes \DOver(N)	.
		\end{split}
	\end{equation}
	Here $VN$ is the vertical bundle of $N$ with respect to some connection on the fibered manifold $N \rightarrow \pr^*(N)$.
\end{prop}
\begin{proof}
	First we note that $\No(M,\pr(N)) \cong \Gamma^\infty( \Lambda^\bullet \left(\factor{TM}{T \pr(N)}\right))$.
				   
	Then, when choosing a torsion free connection on $N$, we get that $\Gamma^\infty(\Sym TN) \cong \DO(N)$. For the vertical operators we get with this $\Gamma^\infty(\Sym (VN)) \cong \DOver(N)$, assuming $\nabla_X \pr^* a$ is the pullback of some function on $\pr(N)$ for any $a\in \CCinf(\pr(N))$ and $X \in \X(N)$.  So we can consider the differential operators as a vector bundle over $N$. 
	In general we have that for two vector bundle  $E,F$  over $N$ we have $\Gamma^\infty(E) \otimes_{\CCinf(N)} \Gamma^\infty(F) \cong \Gamma^\infty(E \otimes F)$. Using this and \cref{th:vecpb} we get the desired result.
\end{proof}

The above proposition shows that one can consider $\HH^\bullet_\diff(\CCinf(M),\DO(N))$ as some sort of multivector fields on $N$, which take as arguments functions on $M$ and have values in the vertical differential operators on $N$.

Finally we want to embed the cohomology as reformulated above  back in to the complex. For this it is necessary to embed the normal bundle of $\pr N$ into the tangent bundle $TM|_{\pr N}$. This can be done for example by choosing a tubular neighborhood.  With this an element $X \otimes D \in 	\Gamma^\infty \left(\pr^\sharp \Lambda^k T \pr N^\bot \right) \otimes \DOver(N)$ can be considered as an element of $\HC_\diff(M,\DO(n))$ by
\begin{equation}
	(X_1 \wedge \dots \wedge X_k \otimes D)(a_1, \dots , a_k)(f) = \sum_{\sigma \in S_k} \sign(\sigma) \sprod{\pr^\sharp\d a_1,X_{\sigma(1)}} \dots  \sprod{\pr^\sharp\d a_k,X_{\sigma(k)}}   D(f).
\end{equation} 
Here $\sprod{\cdot,\cdot}$ denotes the natural pairing between $\pr^\sharp TM|_N$ and $\pr^\sharp T^*M|_N$.

\bibliographystyle{bibstyle}
\bibliography{bimoddef}


\end{document}